\documentclass[11pt,a4paper]{scrartcl}

\usepackage{a4wide}
\usepackage[USenglish]{babel}
\usepackage{latexsym}
\usepackage{amssymb}
\usepackage{mathrsfs}
\usepackage{graphicx}
\usepackage[latin1]{inputenc}
\usepackage{amsmath}
\usepackage{amsfonts}
\usepackage{amsthm}
\usepackage{url}
\usepackage{mathptmx}
\usepackage{helvet}
\usepackage{dsfont}
\usepackage{setspace}
\usepackage{enumerate}
\usepackage{hyperref}
\hypersetup{
  bookmarksopen=true,
  citebordercolor=0 1 0,
  linkbordercolor=1 0 0,
  menubordercolor=1 0 0,
  }

 \newcommand{\bfi}{\begin{fig}}
 \newcommand{\efi}{\end{fig}}

 \newcommand{\btab}{\begin{tab}}
 \newcommand{\etab}{\end{tab}}

 \newcommand{\barr}{\begin{array}}
 \newcommand{\earr}{\end{array}}

 \newcommand{\beqq}{\begin{equation}}
 \newcommand{\eeqq}{\end{equation}}

 \newcommand{\beao}{\begin{eqnarray*}}
 \newcommand{\eeao}{\end{eqnarray*}\noindent}

 \newcommand{\beam}{\begin{eqnarray}}
 \newcommand{\eeam}{\end{eqnarray}\noindent}

 \newcommand{\bdis}{\begin{displaymath}}
 \newcommand{\edis}{\end{displaymath}\noindent}

 \newcommand{\ben}{\begin{enumerate}} 
 \newcommand{\een}{\end{enumerate}} 

 \newcommand{\bali}{\begin{align}} 
 \newcommand{\eali}{\end{align}} 

 \newcommand{\bbn}{\mathbb{N}}
 
 \newcommand{\bbr}{\mathbb{R}}
 \newcommand{\bbc}{\mathbb{C}}

\newcommand{\limp}{\stackrel{\mathbb{P}}{\rightarrow}}

 \newcommand{\limd}{\stackrel{w}{\rightarrow}}

 \newcommand{\limv}{\stackrel{v}{\rightarrow}}
 
 \newcommand{\eqd}{\stackrel{\mathscr{D}}{=}}

 \newcommand{\nto}{{n\to\infty}}

 \newcommand{\scrb}{{\mathscr{B}}}
 \newcommand{\scre}{{\mathscr{E}}}
 \newcommand{\scrf}{{\mathscr{F}}}

 \newcommand{\scrn}{{\mathscr{N}}}

 \newcommand{\al}{{\alpha}}
 \newcommand{\la}{{\lambda}}

 \newcommand{\vep}{{\varepsilon}}
 \newcommand{\ga}{{\gamma}}

 \newcommand{\si}{{\sigma}}
 \newcommand{\Om}{{\Omega}}

 \newcommand{\bbE}{{\mathbb{E}}}
 
 \newcommand{\bbP}{{\mathbb{P}}}
 
 \newcommand{\ov}{\overline}
 \newcommand{\wh}{\widehat}
 \newcommand{\wt}{\widetilde}

\newtheorem{Theorem}{Theorem}[section]
\newtheorem{Corollary}[Theorem]{Corollary}
\newtheorem{Proposition}[Theorem]{Proposition}
\newtheorem{Lemma}[Theorem]{Lemma}
\newtheorem{Definition}[Theorem]{Definition}
\theoremstyle{definition}
\newtheorem{Remark}[Theorem]{Remark}

\AtBeginDocument{%
}

\numberwithin{equation}{section}
\title{Spectral Representation of Multivariate Regularly Varying L\'evy and CARMA Processes}

\author{Florian Fuchs \thanks{TUM Institute for Advanced Study \& Zentrum Mathematik, Technische Universit\"at M\"unchen, Boltzmannstra\ss e 3, D-85748 Garching, Germany. 
\emph{Email: } \ttfamily{ffuchs@ma.tum.de, www-m4.ma.tum.de}} \and Robert Stelzer \thanks{Institute of Mathematical Finance, Ulm University, Helmholtzstra\ss e 18, D-89081 Ulm, Germany.
\emph{Email: } \ttfamily{robert.stelzer@uni-ulm.de, www.uni-ulm.de/mawi/finmath.html}}}

\date{}

\begin{document}
%
%%%%%%%%%%%%%%%%%  Titel und Autoreninformationen  %%%%%%%%%%%%%%%%%
\maketitle
%
%%%%%%%%%%%%%%%%%  Abstract und AMS Classification  %%%%%%%%%%%%%%%%
\begin{abstract}
 A spectral representation for regularly varying L\'evy processes with index between one and two is established and the properties of the resulting random noise
 are discussed in detail giving also new insight in the $L^2$-case where the noise is a random orthogonal measure. 

 This allows a spectral definition of multivariate regularly varying L\'evy-driven continuous time autoregressive moving average (CARMA) processes. It is shown 
 that they extend the well-studied case with finite second moments and coincide with definitions previously used in the infinite variance case when they apply. 
\end{abstract}
\vspace{0.5cm}
\noindent
\begin{tabbing}
\> Secondary: \, 62M15, 60G57
\end{tabbing}
\vspace{0.5cm}\noindent\emph{Keywords:} CARMA process, L\'evy process, spectral representation, multivariate regular variation, random noise, random orthogonal measure
%
%%%%%%%%%%%%%%%%%  Einleitung  %%%%%%%%%%%%%%%%%%%%%%%%%%%%%%%%%%%%%
\section{Introduction}
Being the continuous time analog of the well-known ARMA processes (see e.g. \cite{Brockwelletal1991}), 
continuous time ARMA (CARMA) processes have been extensively studied over the recent years (see e.g. \cite{Brockwell2001, Brockwell2004, Brockwelletal2009, Tauchenetal2006} and
references therein) and widely used in various areas of application like engineering and finance (e.g. \cite{Larssonetal2006, Tauchenetal2006}). 
The advantage of continuous time modeling is that it allows handling irregularly spaced time series and in particular high frequency data 
often appearing in finance. Originally, driving processes of CARMA models were restricted to Brownian motion; however, \cite{Brockwell2001} allowed for L\'evy processes 
which have a finite $r$-th moment for some $r>0$.

In practice multivariate models are necessary in many applications in order to take account of the joint behavior of several times series. 
The multivariate version of the CARMA process (MCARMA) has been introduced in \cite{Marquardtetal2007} where an explicit construction using a state space representation and a spectral 
representation of the driving L\'evy process in the $L^2$-case is given.

For the analysis of many statistical and probabilistic problems in conjunction with various stochastic processes, a significant tool is often provided by the spectral representations of these processes. 
For instance, the spectral representations of symmetric stable processes have successfully been used to solve prediction and interpolation problems 
(see e.g. \cite{Cambanisetal1989b, Hosoya1982}) and to study structural and path properties for certain subclasses of these processes (see e.g. \cite{Cambanisetal1987, Rootzen1978}).

However, in \cite{Marquardtetal2007} a spectral representation of MCARMA processes is only obtained under the assumption that the driving L\'evy process has finite second moments. 
On the contrary, there are important applications  where it seems to be adequate to relax that assumption, see e.g. \cite{Garciaetal2010}, where a stable CARMA$(2,1)$ model is fitted 
to spot prices from the Singapore New Electricity Market. 

The aim of this paper is to introduce multivariate CARMA processes that are driven by a regularly varying L\'evy process and to establish a spectral representation for them.
The latter will be derived from a spectral representation of the underlying L\'evy process which is, apart from the fact that we deal with regularly varying processes which 
are a generalization of $\al$-stable processes, a main difference to the works by Cambanis, Houdr\'e, Makagon, Mandrekar and Soltani 
(see \cite{Cambanisetal1993, Cambanisetal1984, Makagonetal1990}) where spectral representations are deduced directly for moving averages of the underlying stable process. 
Furthermore, we study in detail properties of the corresponding random noise for which so far only existence has been addressed in the literature to the best of our knowledge. 
In this connection we are going to prove that the increments of the noise are neither independently nor stationarily scattered giving also new insight in the $L^2$-case 
where the corresponding random orthogonal measure has hence always (except in the purely Brownian setting) uncorrelated but dependent increments. 
Moreover, it is shown that the random noise inherits moments exactly and always has a L\'evy measure around zero with infinite activity. 
Finally, if the underlying L\'evy process has a moment generating function in a neighborhood of zero then so does its corresponding noise. 

The remainder of this paper is organized as follows. In Section 2 we start with a brief overview of notation and then give a summary of the concept of multivariate regular variation.
The third section derives a spectral representation of regularly varying L\'evy processes followed by a detailed discussion of the properties of the resulting random noise. 
Thereafter, a spectral definiton of multivariate regularly varying CARMA processes is given in the fourth section. We explain that in a sense this spectral representation 
(in the summability sense) is optimal, i.e. it cannot be improved in general to a bona fide spectral representation. 
The last section shows consistency of our definition with the so-called causal MCARMA processes, introduced in \cite{Marquardtetal2007}. 
A brief summary of some results for Fourier transforms on the real line necessary for our proofs can be found in the appendix. 
%
%%%%%%%%%%%%%%%%%  Preliminaries %%%%%%%%%%%%%%%%%%%%%%%%%%%%%%%%%%%
\section{Preliminaries}
\subsection{Notation}
%
%Zahlen, Vektoren und Matrizen
Given the real numbers $\bbr$ we use the convention $\bbr_+:=\left.\left(0,\infty\right.\right)$. For the minimum of two real numbers $a,b\in\bbr$ we write shortly $a\wedge b$.
The real and imaginary part of a complex number $z\in\bbc$ is written as $\mathrm{Re}(z)$ and $\mathrm{Im}(z)$, respectively. 
The set of $n\times d$ matrices over the field $\mathbb{K}$ is denoted by $M_{n\times d}(\mathbb{K})$, where $\mathbb{K}\in\left\{\bbr, \bbc\right\}$. 
We set $M_d(\mathbb{K}):=M_{d\times d}(\mathbb{K})$ and define $\mathbb{S}_d(\mathbb{K})$ as the linear subspace of symmetric and Hermitian matrices in the real and complex case, respectively. 
The positive semidefinite cone is denoted by $\mathbb{S}_d^+(\mathbb{K})$, the transpose of $A\in M_{n\times d}(\bbr)$ is written as $A^\prime$, the complex conjugate transpose 
of $A\in M_{n\times d}(\bbc)$ as $A^*$ and the identity matrix in $M_d(\mathbb{K})$ shall be denoted by $\mathrm{I}_d$.

%Skalarprodukte, Normen, Topologien, Lebesguemaß, Hilberträume und Banachräume
On $\mathbb{K}\in\left\{\bbr, \bbc\right\}$ the Euclidean norm is denoted by $\left|\hspace{0.5mm}\cdot\hspace{0.5mm}\right|$ whereas on $\mathbb{K}^d$ it will be usually 
written as $\left\|\hspace{0.5mm}\cdot\hspace{0.5mm}\right\|$. 
Recall the fact that two norms on a finite dimensional linear space are always equivalent which is why our results remain true if we replace the Euclidean norm by any other norm.
A scalar product on linear spaces is written as $\langle\hspace{0.5mm}\cdot\hspace{0.5mm},\hspace{0.5mm}\cdot\hspace{0.5mm}\rangle$; in $\bbr^d$ and $\bbc^d$, we again usually take the Euclidean one. 
If $X$ and $Y$ are normed linear spaces, let $B(X,Y)$ be the set of bounded linear operators from $X$ into $Y$. On $B(X,Y)$ we will usually use the operator norm which, 
in the case of $Y$ being a Banach space, turns $B(X,Y)$ itself into a Banach space. In particular we always equip $M_{n\times d}(\mathbb{K})=B(\mathbb{K}^d,\mathbb{K}^n)$ 
with the corresponding operator norm if not stated otherwise. If $X$ is a topological space, we denote by $\scrb(X)$ the Borel $\si$-algebra on $X$, that is the smallest $\si$-algebra on $X$ containing all open subsets
of $X$. The Lebesgue measure on $(\bbr^d, \scrb(\bbr^d))$ is written as $\la^d$. 

%Stochastische Notationen
The collection of all $\mathbb{K}^d$-valued, $\mathbb{K}\in\left\{\bbr, \bbc\right\}$, random variables defined on some probability space 
$(\Om, \scrf, \bbP)$ is written as $L^0(\Om, \scrf, \bbP; \mathbb{K}^d)$. 
For two random variables $X$ and $Y$ the notation $X\eqd Y$ means equality in distribution. 
For $X\in L^0(\Om, \scrf, \bbP; \mathbb{K}^d)$ we say that $X\in L^p(\Om, \scrf, \bbP; \mathbb{K}^d),\,1\leq p<\infty$, if $\bbE[\left\|X\right\|^p]$ is finite. 
If we define the norm $\left\|X\right\|_{L^p}:=(\bbE[\left\|X\right\|^p])^{1/p}$ and, as usual, do not distinguish between random variables and equivalence classes of random variables, 
$L^p(\Om, \scrf, \bbP; \mathbb{K}^d)$ becomes a Banach space. If we consider a sequence of random variables $(X_n)_{n\in\bbn}$, we shall denote convergence in probability of the sequence to some 
random variable $X$ by $X_n\limp X$ and convergence in distribution by $X_n\limd X$.
\subsection{Multivariate Regular Variation}
For the analysis of the tail behavior of stochastic processes, the concept of regular variation is well established. 
For detailed introductions into the different approaches of multivariate regular variation, we refer the reader to \cite{Lindskog2004} and \cite{Resnick2007}.

We start with a definition from \cite{Hultetal2006}.  Let therefore $\limv$ denote vague convergence. It is defined on the one-point uncompactification $\ov{\bbr}^d\backslash\left\{0\right\}$ 
(where $\ov{\bbr}:=[-\infty,\infty]$), which assures that the Borel sets of $\bbr^d$ that are bounded away form the origin can be referred to as the relatively compact sets in the vague topology. 
\begin{Definition}[\textbf{Multivariate Regular Variation}]$~~$
 \begin{enumerate}[(i)]
 \item
 An $\bbr^d$-valued random vector $X$ is called {\rm regularly varying} with index $\al>0$, if there exist a function $l:\bbr\to\bbr$ which is slowly varying at infinity and 
 a non-zero Radon measure $\kappa$ defined on $\scrb(\ov{\bbr}^d\backslash\left\{0\right\})$ with $\kappa(\ov{\bbr}^d\backslash\bbr^d)=0$ such that, as $u\to\infty$, 
 \[u^\al l(u)\bbP(u^{-1}X\in\hspace{0.5mm}\cdot\hspace{0.5mm})\limv\kappa(\hspace{0.5mm}\cdot\hspace{0.5mm})\]
 on $\scrb(\ov{\bbr}^d\backslash\left\{0\right\})$. We write $X\in RV(\al, l, \kappa)$. 
 \item
 Similarly, we call a Radon measure $\nu$ regularly varying, if $\al$, $l$ and $\kappa$ exist as above such that
 $u^\al l(u)\nu(u\hspace{0.5mm}\cdot\hspace{0.5mm})\limv\kappa(\hspace{0.5mm}\cdot\hspace{0.5mm})$ on $\scrb(\ov{\bbr}^d\backslash\left\{0\right\})$
 as $u\to\infty$ and we write $\nu\in RV(\al, l, \kappa)$. 
 \item 
 We say that a complex random vector $X=(X_1,\ldots,X_d)$ is regularly varying if and only if the real random vector
 $(\mathrm{Re}(X_1),\,\mathrm{Im}(X_1),\ldots,\,\mathrm{Re}(X_d),\,\mathrm{Im}(X_d))$ is regularly varying. 
 \item
 A $d$-dimensional stochastic process $(X_t)_{t\in\bbr}$ is called regularly varying with index $\al$ if all its finite dimensional 
 distributions are regularly varying with index $\al$. 
\end{enumerate}
\label{Definition 6.1}
\end{Definition}
Since we consider L\'evy-driven CARMA processes, we recall that a two-sided L\'evy process $L=(L_t)_{t\in\bbr}$ in $\bbr^d$ is determined by its characteristic function in the 
L\'evy-Khintchine form $\bbE\left[e^{i\langle z, L_t\rangle}\right]=\exp{\left\{\left|t\right|\cdot\psi_L(\mathrm{sgn}(t)z)\right\}},\,t\in\bbr$, with
\beqq
 \psi_L(z)=i\langle\ga,z\rangle-\frac{1}{2}\langle z,\Sigma z\rangle+
 \int_{\bbr^d}\left(e^{i\langle z,x\rangle}-1-i\langle z,x\rangle\mathds{1}_{[0,1]}(\left\|x\right\|)\right)\nu(dx),\quad z\in\bbr^d,
\label{Equation 1.2}
\eeqq
where $\ga\in\bbr^d,\,\Sigma\in\mathbb{S}_d^+(\bbr)$ and $\nu$ is a measure on $(\bbr^d, \scrb(\bbr^d))$ satisfying
$\int_{\bbr^d}\big(1\wedge\left\|x\right\|^2\big)\nu(dx)<\infty$ and $\nu(\left\{0\right\})=0$.
The triplet $(\ga, \Sigma, \nu)$ is referred to as the generating triplet of the L\'evy process, $\nu$ is said to be the L\'evy measure of $L$
and $\mathds{1}_{[0,1]}(\left\|x\right\|)$ is called truncation function. The same representation is true for the characteristic function of any infinitely divisible distribution. 
A general introduction to L\'evy processes and infinitely divisible distributions can be found in \cite{Sato1999}. 

The following very useful connection between regular variation of an infinitely divisible random variable and its L\'evy measure exists.
\begin{Theorem}[cf. \cite{Hultetal2006}, Proposition 3.1]$~~$\\
 Let $X$ be an infinitely divisible $\bbr^d$-valued random vector with L\'evy measure $\nu$. Then $X\in RV(\al, l, \kappa)$ if and only if $\nu\in RV(\al, l, \kappa)$. 
\label{Theorem 6.2}
\end{Theorem}
\noindent
It is then easy to show that analogously a $d$-dimensional L\'evy process $L=(L_t)_{t\in\bbr}$ is regularly varying of index $\al$ if and only if its L\'evy measure $\nu$ 
is regularly varying of index $\al$. Strictly speaking the L\'evy process $L$ is not regularly varying since $L_0\equiv 0$ a.s., but, 
as all other finite dimensional margins are regularly varying, we neglect that inaccuracy. 
%
%%%%%%%%%%%%%%%%  Reg. variierende Lévy Prozesse  %%%%%%%%%%%%%%%%%
\section{Multivariate Regularly Varying L\'evy Processes}
In this section we discuss multivariate L\'evy processes which are regularly varying with index between one and two. We shall derive
a spectral representation and discuss properties of the associated random noise.
\subsection{Spectral Representation of Regularly Varying L\'evy Processes}
Let $\scre(\bbr)$ denote the collection of all elementary subsets of $\bbr$, i.e. the ring generated by the semi-ring of half-open bounded intervals $[a,b)$ with $-\infty<a<b<\infty$. 
Since we derive a spectral representation in the summability sense, it will be sufficient to define the associated random noise on $\scre(\bbr)$. 
We call the arising dependently scattered, additive random noises ``random contents'' in order to place emphasis on the fact that we do not necessarily have $\si$-additive set functions.
\begin{Definition}[\textbf{Regularly Varying Random Content}]
 For $\al\in(1,2]$ a $d$-dimensional {\rm regularly varying random content} with index $\al$ is a set function $M:\scre(\bbr)\to L^0(\Om, \scrf, \bbP; \bbc^d)$ satisfying 
 \ben[(i)]
  \item
   $M(A)$ is a complex $d$-dimensional random vector that is regularly varying with index $\al$ for all $A\in\scre(\bbr)$,
  \item
   $M(\bigcup_{i=1}^n A_i)=\sum_{i=1}^n M(A_i)$ a.s. whenever $A_1,\ldots,\,A_n\in\scre(\bbr)$ are pairwise disjoint (i.e. $M$ is additive).  
 \een
\label{Definition 6.3}
\end{Definition}
Integration of simple functions $f=\sum_{i=1}^n f_i\mathds{1}_{A_i}$ (with $f_i\in M_d(\bbc),\,i=1,\ldots,n,\,n\in\bbn$, and $A_i\in\scre(\bbr)$ mutually disjoint) with respect to $M$ is defined by
\[\int_\bbr f\,dM := \sum\limits_{i=1}^n f_i\,M(A_i)\]
which is obviously a complex $d$-dimensional random vector. The integral is linear for simple functions and it is well-defined due to the additivity of $M$. 

In order to extend integration to a more general class of integrands the following theorem will be crucial. Therefore, we define the set 
\[L^\delta(M_{k\times d}(\bbr)):=\left\{f:\bbr\to M_{k\times d}(\bbr)\text{ measurable},\ \int_\bbr\left\|f(t)\right\|^\delta dt<\infty\right\}.\]
\begin{Theorem}
 Let $L=(L_t)_{t\in\bbr}$ be a $d$-dimensional L\'evy process with $\bbE[L_1]=0$ and generating triplet $(\ga,\Sigma,\nu)$ where $\nu$ is regularly varying with index $\al\in(1,2]$. 
 Let $f:\bbr\to M_{k\times d}(\bbr)$ be measurable and $f_n:\bbr\to M_{k\times d}(\bbr)$ be a sequence of measurable functions such that $f_n\to f$ as $\nto$ in $L^\delta(M_{k\times d}(\bbr))$ 
 for some $\delta<\al$. Moreover, assume that $\left\|f_n(s)-f(s)\right\|+\left\|f(s)\right\|\leq C$ for all $n\in\bbn,\,s\in\bbr$ and some constant $C>0$. Then the sequence of integrals 
 $\int_\bbr f_n\,dL$ converges in probability to $\int_\bbr f\,dL$ as $\nto$.
\label{Theorem 6.7}
\end{Theorem}
\noindent
Note that this continuity result for integrals with respect to L\'evy processes is of general interest of its own. Before we pass on to the proof, we recall a result regarding the existence of these integrals. 
\begin{Theorem}[cf. \cite{Moseretal2010}, Theorem 2.5]$~~$\\
 Let $L=(L_t)_{t\in\bbr}$ be a $d$-dimensional L\'evy process with generating triplet $(\ga,\Sigma,\nu)$, let $\nu$ be regularly varying with index $\al\in(1,2]$ and let $f:\bbr\to M_{k\times d}(\bbr)$ be measurable. 
 Then $f$ is $L$-integrable in the sense of Rajput and Rosi\'nski \cite{Rajputetal1989} if it is bounded, $\bbE[L_1]=0$ and $f\in L^{\delta}(M_{k\times d}(\bbr))$ for some $\delta<\al$.
\label{Theorem 6.5}
\end{Theorem}
\begin{proof}[\textbf{Proof of Theorem \ref{Theorem 6.7}.}]$~~$\\
 Note first that the integrals $\int_\bbr f_n\,dL$ and $\int_\bbr f\,dL$ are well-defined due to Theorem \ref{Theorem 6.5}. Letting $g_n:=f_n-f$, we have to show that $\int_\bbr g_n\,dL\limp 0$ 
 which is equivalent to $\int_\bbr g_n\,dL\limd 0$.
 
 Now the distribution of every $\int_\bbr g_n\,dL$ is infinitely divisible and possesses the generating triplet $(\ga_n,\Sigma_n,\nu_n)$ given by 
 (cf. \cite[Theorem 2.7]{Rajputetal1989} and \cite[Theorem 2.4]{Moseretal2010} for a multivariate extension) 
 \[\ga_n = \int_\bbr\left(g_n(s)\ga + \int_{\bbr^d} g_n(s)x \left(\mathds{1}_{[0,1]}(\left\|g_n(s)x\right\|)-\mathds{1}_{[0,1]}(\left\|x\right\|)\right)\nu(dx)\right)ds,\]
 \[\Sigma_n = \int_\bbr g_n(s)\Sigma\,g_n(s)^\prime ds\qquad\text{and}\qquad \nu_n(B)=\int_\bbr\int_{\bbr^d}\mathds{1}_B(g_n(s)x)\,\nu(dx)\,ds,\quad B\in\scrb(\bbr^k_*),\]
 where $\bbr^k_*:=\bbr^k\backslash\left\{0\right\}$ denotes the punctured Euclidean space.

 In order to use \cite[Theorem 8.7]{Sato1999}, we change the truncation function in (\ref{Equation 1.2}) from $\mathds{1}_{[0,1]}(\left\|x\right\|)$ to the continuous truncation function 
 $c(x):=\mathds{1}_{[0,1]}(\left\|x\right\|)+\mathds{1}_{(1,2]}(\left\|x\right\|)(2-\left\|x\right\|)$. Consequently (cf. \cite[Remark 8.4]{Sato1999}) the generating triplet 
 of $\int_\bbr g_n\,dL$ changes to $(\ga_{n,c},\Sigma_n,\nu_n)_c$ where 
 \[\ga_{n,c} = \ga_n + \int_{\bbr^k} x\left(c(x)-\mathds{1}_{[0,1]}(\left\|x\right\|)\right)\nu_n(dx) = \ga_n + \int_{\left\{\left\|x\right\|\in(1,2]\right\}} x\,(2-\left\|x\right\|)\,\nu_n(dx).\]
 The remainder of the proof is dedicated to the verification of conditions (1) - (3) in \cite[Theorem 8.7]{Sato1999}.
 
 To this end we first show that 
 \beqq
  \int_\bbr\int_{\bbr^d}\left(1\wedge\left\|g_n(s)x\right\|^2\right)\nu(dx)\,ds\stackrel{\nto}{\to} 0.
  \label{Equation 6.1}
 \eeqq
 We get
 \begin{align}
  \int_\bbr&\int_{\bbr^d}\left(1\wedge\left\|g_n(s)x\right\|^2\right)\,\nu(dx)\,ds \nonumber \\
  &\qquad=\int_\bbr\int_{\bbr^d}\mathds{1}_{\left\{\left\|g_n(s)x\right\|>1\right\}}\,\nu(dx)\,ds + \int_\bbr\int_{\bbr^d}\left\|g_n(s)x\right\|^2\mathds{1}_{\left\{\left\|g_n(s)x\right\|\leq 1\right\}}\,\nu(dx)\,ds.
 \label{Equation 6.2}
 \end{align}
 Arguing in an analogous manner as in the proof of \cite[Theorem 2.5]{Moseretal2010} and using the assumption that $g_n$ converges to $0$ in $L^\delta(M_{k\times d}(\bbr))$ we deduce that the first term on the
 right-hand side of (\ref{Equation 6.2}) converges to $0$ as $\nto$. The second term on the right hand side of (\ref{Equation 6.2}) can be bounded by
 \begin{align}
  \int_\bbr&\int_{\bbr^d}\left\|g_n(s)x\right\|^2\mathds{1}_{\left\{\left\|g_n(s)x\right\|\leq 1\right\}}\,\nu(dx)\,ds \nonumber \\
   &\leq\int_\bbr\left\|g_n(s)\right\|^2 ds\int_{\left\{\left\|x\right\|<1\right\}}\left\|x\right\|^2 \nu(dx) + \int_\bbr\left\|g_n(s)\right\|^\delta ds\int_{\left\{\left\|x\right\|\geq 1\right\}}\left\|x\right\|^\delta \nu(dx)
    \stackrel{\nto}{\to} 0 \nonumber
 \end{align}
 where we used the fact that the boundedness of the sequence $g_n$ together with the convergence to $0$ in $L^\delta(M_{k\times d}(\bbr))$ implies that $g_n$ converges to $0$ also in $L^2(M_{k\times d}(\bbr))$.
 Moreover, note that the integral $\int_{\left\{\left\|x\right\|\geq 1\right\}}\left\|x\right\|^\delta \nu(dx)<\infty$ due to \cite[Corollary 25.8]{Sato1999}, since $0<\delta<\al$ and hence the L\'evy process has a finite
 $\delta$-th moment. Thus (\ref{Equation 6.1}) is shown. 
 
 Let us now verify condition (1) of \cite[Theorem 8.7]{Sato1999}. Therefore, we show that $\nu_n$ converges in total variation to the zero measure outside of any fixed neighborhood of $0$. 
 Indeed, we obtain for any $U_\vep=\left\{x\in\bbr^k:\,\left\|x\right\|<\vep\right\}$ with $\vep>0$
 \begin{align}
  \nu_n(\bbr^k\backslash U_\vep)&=\la^1\otimes\nu\left(\left\{(s,x)\in\bbr\times\bbr^d:\,\left\|g_n(s)x\right\|\geq\vep\right\}\right) 
  \leq\frac{1}{1\wedge\vep^2} \int_\bbr\int_{\bbr^d}\left(1\wedge\left\|g_n(s)x\right\|^2\right)\nu(dx)\,ds \nonumber 
 \end{align} 
 where the right-hand side converges to $0$ by virtue of (\ref{Equation 6.1}). 
 
 As to condition (2), note first that $\left\|\Sigma_n\right\|\leq\left\|\Sigma\right\|\cdot\int_\bbr\left\|g_n(s)\right\|^2 ds \stackrel{\nto}{\to} 0$ since $g_n$ converges to $0$ in $L^2(M_{k\times d}(\bbr))$ 
 as previously noted. Hence, using again (\ref{Equation 6.1}) we obtain for any $\vep\in(0,1)$, 
 \begin{align}
  \left|\langle z, \Sigma_{n,\vep}z\rangle\right|&\leq\left|\langle z, \Sigma_n z\rangle\right| + \int_{\left\{\left\|y\right\|\leq\vep\right\}}\langle z, y\rangle^2\,\nu_n(dy) \nonumber \\
  &=\left|\langle z, \Sigma_n z\rangle\right| + \int_{\left\{(s,x)\in\bbr\times\bbr^d:\,\left\|g_n(s)x\right\|\leq\vep\right\}}\langle z, g_n(s)x\rangle^2\,(\la^1\otimes\nu)(d(s,x)) \nonumber \\
  &\leq\left\|z\right\|^2\left(\left\|\Sigma_n\right\|+\int_\bbr\int_{\bbr^d}\left(1\wedge\left\|g_n(s)x\right\|^2\right)\nu(dx)\,ds\right)\stackrel{\nto}{\to}0. \nonumber
 \end{align}
 This in particular yields $\lim\limits_{\vep\searrow 0}\limsup\limits_\nto\left|\langle z, \Sigma_{n,\vep}z\rangle\right| = 0$ for all $z\in\bbr^k$. 
 
 Finally we show condition (3), i.e. $\ga_{n,c}\to 0$ as $\nto$. We immediately obtain that 
 \[\int_{\left\{\left\|x\right\|\in(1,2]\right\}} \left\|x\right\| (2-\left\|x\right\|)\,\nu_n(dx)\leq\nu_n\left(\left\{\left\|x\right\|\in(1,2]\right\}\right)\stackrel{\nto}{\to}0\]
 since $\nu_n$ converges in total variation to the zero measure outside of any fixed neighborhood of $0$. 
 The assumption $\bbE[L_1]=0$ implies $\ga = -\int_{\left\{\left\|x\right\|>1\right\}} x\,\nu(dx)$ (cf. \cite[Example 25.12]{Sato1999}). Thus we can choose any $\xi\in(\delta,\al),\,\xi>1$ and get
 \begin{align}
  \left\|\ga_n\right\|&\leq\int_\bbr\bigg\|g_n(s)\ga + \int_{\left\{\left\|x\right\|>1\right\}} g_n(s)x\,\mathds{1}_{\left\{\left\|g_n(s)x\right\|\leq 1\right\}}\,\nu(dx)
   - \int_{\left\{\left\|x\right\|\leq 1\right\}} g_n(s)x\,\mathds{1}_{\left\{\left\|g_n(s)x\right\|>1\right\}}\,\nu(dx)\bigg\| ds \nonumber \\
   &\leq\int_\bbr\int_{\bbr^d}\left\|g_n(s)x\right\|\mathds{1}_{\left\{\left\|g_n(s)x\right\|>1\right\}}\,\nu(dx)\,ds
    \leq\int_\bbr\int_{\bbr^d}\left\|g_n(s)x\right\|^\xi\mathds{1}_{\left\{\left\|x\right\|\geq\frac{1}{C}\right\}}\,\nu(dx)\,ds \nonumber \\
   &\leq C^{\xi-\delta}\int_\bbr\left\|g_n(s)\right\|^\delta ds \int_{\left\{\left\|x\right\|\geq\frac{1}{C}\right\}}\left\|x\right\|^\xi\nu(dx)\stackrel{\nto}{\to}0 \nonumber
 \end{align}
 since $g_n\to 0$ in $L^\delta(M_{k\times d}(\bbr))$. Note again that $\int_{\left\{\left\|x\right\|\geq\frac{1}{C}\right\}}\left\|x\right\|^\xi\nu(dx)$ is finite since $0<\xi<\al$ and hence 
 the underlying L\'evy process has a finite $\xi$-th moment. Together this shows $\ga_{n,c}\to 0$ as $\nto$. 
 
 Now to conclude the proof we can use \cite[Theorem 8.7]{Sato1999} which yields $\int_\bbr g_n\,dL\limd 0$.
\end{proof}
The following theorem is our first main result establishing a spectral representation in the summability sense of regularly varying L\'evy processes.
\begin{Theorem}
 Let $L=(L_t)_{t\in\bbr}$ be a L\'evy process in $\bbr^d$, regularly varying of index $\al\in(1,2]$ and suppose $\bbE[L_1]=0$. Then there is a regularly varying random content 
 $M:\scre(\bbr)\to L^0(\Om, \scrf, \bbP; \bbc^d)$ with index $\al$ such that 
 \[L_t = \mathop{\bbP-\lim}\limits_{\la\to\infty}\int_{-\la}^\la \frac{e^{it\mu}-1}{i\mu}\cdot\left(1-\frac{\left|\mu\right|}{\la}\right) M(d\mu), \quad t\in\bbr,\]
 where $\mathop{\bbP-\lim}$ denotes the limit in probability. The random content $M$ is given by
 \[M(A)=\frac{1}{\sqrt{2\pi}}\int_{-\infty}^\infty\wh{\mathds{1}_A}(\mu)\,L(d\mu),\quad A\in\scre(\bbr),\]
 where $\wh{\mathds{1}_A}$ is the Fourier transform of $\mathds{1}_A$ (see appendix). 
\label{Theorem 6.4}
\end{Theorem}
\begin{proof}[\textbf{Proof.}]$~~$\\
 \textbf{Step 1}: We first show that $M$ is well-defined and a regularly varying random content on $\scre(\bbr)$. For $-\infty<a<b<\infty$ we obtain 
 \begin{align}
  \wh{\mathds{1}_{[a,b)}}(\mu)&=\frac{1}{\sqrt{2\pi}}\cdot\frac{e^{-ia \mu}-e^{-ib \mu}}{i\mu}
   =\frac{1}{\sqrt{2\pi}}\left(\frac{\sin\left(b\mu\right)-\sin\left(a\mu\right)}{\mu}+i\cdot\frac{\cos\left(b\mu\right)-\cos\left(a\mu\right)}{\mu}\right),\ \mu\in\bbr, \nonumber
 \end{align}
 which is obviously a bounded element of $L^\delta(\bbc)$ for arbitrary $\delta>1$. This implies that, for any $A\in\scre(\bbr)$ and any $\delta>1$, the Fourier transform $\wh{\mathds{1}_A}$ is bounded and in
 $L^\delta(\bbc)$ and hence $M$ is well-defined by virtue of Theorem \ref{Theorem 6.5}. A simple application of \cite[Theorem 3.2]{Moseretal2010} shows that $M$ is a regularly varying random content 
 with index $\al$. 
 
 \noindent 
 \textbf{Step 2}: Next we want to study integration of more general than simple functions with respect to $M$. For simple functions $f=\sum_{i=1}^n f_i\mathds{1}_{A_i}$ we deduce, 
 using the linearity of the Fourier transformation, the identity
 \beqq
  \int_{-\infty}^\infty f\,dM = \sum\limits_{i=1}^n f_i\,M(A_i) = \frac{1}{\sqrt{2\pi}}\int_{-\infty}^\infty\sum\limits_{i=1}^n f_i\wh{\mathds{1}_{A_i}}\,dL = \frac{1}{\sqrt{2\pi}}\int_{-\infty}^\infty\wh{f}\,dL.
 \label{Equation 6.4}
 \eeqq

 If now $f:\bbr\to M_d(\bbc)$ is an element of $L^p(M_d(\bbc))$ for some $p\in[1,2]$ such that there is a sequence of $\scre(\bbr)$-simple functions $f_n$ satisfying
 \beqq
  \wh{f_n}\stackrel{L^\delta(M_d(\bbc))}{\to}\wh{f}\quad\text{as }\nto\quad\text{for some }\delta<\al\quad\text{and}
  \label{Equation 4.10}
 \eeqq
 \[\big\|\wh{f_n}(\mu)-\wh{f}(\mu)\big\|+\big\|\wh{f}(\mu)\big\|\leq C\quad\text{for all }n\in\bbn,\,\mu\in\bbr\text{ and some constant }C>0\]
 (recall that for almost all $\mu\in\bbr$ the Fourier transform $\wh{f}(\mu)$ is equal to $\lim_{k\to\infty}\frac{1}{\sqrt{2\pi}}\int_{-n_k}^{n_k} e^{-ix\mu} f(x)\,dx$ for a suitably chosen subsequence
 $(n_k)_{k\in\bbn}$, otherwise we use the convention $\wh{f}(\mu)=0$), then we define the integral $\int_{-\infty}^\infty f\,dM$ as the limit in probability of the sequence of simple integrals 
 $\int_{-\infty}^\infty f_n\,dM$. 

 Note that this sequence of integrals is well-defined since every $f_n$ is $\scre(\bbr)$-simple. Since we can always identify $\bbc$ with $\bbr^2$ and $\bbc^d$ with $\left(\bbr^2\right)^d$ and 
 since the multiplication of two complex numbers $x=x_1+ix_2$ and $y=y_1+iy_2$ can be regarded as the (real) matrix-vector multiplication 
 \[\begin{pmatrix}x_1 & -x_2 \\ x_2 & x_1\end{pmatrix}\cdot\begin{pmatrix} y_1 \\ y_2 \end{pmatrix},\]
 it is easy to see that Theorem \ref{Theorem 6.7} holds with functions that take values in the complex $k\times d$ matrices as well. Thus we have 
 \beqq
  \int_{-\infty}^\infty\wh{f_n}\,dL\limp\int_{-\infty}^\infty\wh{f}\,dL
  \label{Equation 6.5}
 \eeqq
 as $\nto$. Using (\ref{Equation 6.4}), we know that $\int_{-\infty}^\infty f_n\,dM = \frac{1}{\sqrt{2\pi}}\int_{-\infty}^\infty\wh{f_n}\,dL$ and hence the sequence of simple integrals $\int_{-\infty}^\infty f_n\,dM$
 converges in probability which shows that $\int_{-\infty}^\infty f\,dM$ is well-defined. Moreover, (\ref{Equation 6.5}) immediately yields 
 \beqq
  \int_{-\infty}^\infty f\,dM = \frac{1}{\sqrt{2\pi}}\int_{-\infty}^\infty\wh{f}\,dL.
  \label{Equation 6.6}
 \eeqq
 We shall call such functions $M$-integrable.
 
 \noindent
 \textbf{Step 3}: Let us now define, for any $-\infty<a<b<\infty$,   
 \[f(\mu):=\frac{e^{ib\mu}-e^{ia\mu}}{i\mu}\quad\text{and}\quad\Phi_\la(\mu):=\left(1-\frac{\left|\mu\right|}{\la}\right)\mathds{1}_{[-\la,\la]}(\mu),\quad\mu\in\bbr.\]
 Then $f\cdot\Phi_\la:\bbr\to\bbc$ is continuous with compact support on $\bbr$. Moreover, note that $f\Phi_\la$ is $M$-integrable. For, writing 
 \[f(\mu)\cdot\Phi_\la(\mu)=\left(\frac{\sin{(b\mu)}-\sin{(a\mu)}}{\mu}-i\cdot\frac{\cos{(b\mu)}-\cos{(a\mu)}}{\mu}\right)\cdot\left(1-\frac{\left|\mu\right|}{\la}\right)
 \mathds{1}_{[-\la,\la]}(\mu),\quad \mu\in\bbr,\]
 one immediately verifies that there is a lower sequence of $\scre(\bbr)$-simple functions such that
 \[\left|f_n\right|\leq\left|f\Phi_\la\right|,\ \mathrm{Var}(f_n)\leq\mathrm{Var}(f\Phi_\la)<\infty\quad\text{and}\quad f_n\stackrel{L^1(\bbc)}{\to} f\Phi_\la\,\text{as }\nto\]
 where $\mathrm{Var}(\hspace{0.5mm}\cdot\hspace{0.5mm})$ denotes the total variation (cf. \cite[proof of Theorem 3.1]{Cambanisetal1993}). We show that $\wh{f_n}\to\wh{f\Phi_\la}$ in $L^\delta(\bbc)$ 
 as $\nto$ for any $\delta\in(1,\al)$. We have
 \begin{align}
  \big\|\wh{f_n}-\wh{f\Phi_\la}\big\|_{L^\delta}^\delta
   &=\int_{-1}^1\big|\wh{f_n}(\mu)-\wh{f\Phi_\la}(\mu)\big|^\delta d\mu + \int_{\left\{\left|\mu\right|>1\right\}}\big|\wh{f_n}(\mu)-\wh{f\Phi_\la}(\mu)\big|^\delta d\mu \nonumber \\
   &\leq 2\cdot\left\|f_n-f\Phi_\la\right\|_{L^1}^\delta + \int_{\left\{\left|\mu\right|>1\right\}}\big|\wh{f_n}(\mu)-\wh{f\Phi_\la}(\mu)\big|^\delta d\mu \label{Equation 4.12}
 \end{align}
 where the first addend vanishes as $\nto$. Integration by parts yields
 \begin{align}
  \big|\wh{f_n}(\mu)\big|
   &\leq\Big|\int_\bbr f_n(s) e^{-i\mu s}\,ds\Big| = \Big|\int_\bbr f_n(s)\,d\Big(-\frac{1}{i\mu}e^{-i\mu s}\Big)(s)\Big|
    \leq \frac{1}{\left|\mu\right|}\Big(2\sup\limits_s\left|f_n(s)\right|+\mathrm{Var}(f_n)\Big) \nonumber \\
   &\leq\frac{1}{\left|\mu\right|}\Big(2\sup\limits_s\left|f(s)\Phi_\la(s)\right| + \mathrm{Var}(f\Phi_\la)\Big) \nonumber
 \end{align}
 and since $\big|\wh{f_n}(\mu)-\wh{f\Phi_\la}(\mu)\big|\leq\left\|f_n-f\Phi_\la\right\|_{L^1}\stackrel{\nto}{\to} 0$ for all $\mu\in\bbr$, we obtain, due to the Dominated Convergence Theorem, 
 that the second term in (\ref{Equation 4.12}) vanishes as well as $\nto$. The additional boundedness condition is obvious and hence $f\Phi_\la$ is indeed $M$-integrable. 
 
 \noindent
 \textbf{Step 4}: We set $g(\mu):=\sqrt{2\pi}\mathds{1}_{[a,b)}(\mu)$ and $h(\mu):=g(-\mu),\,\mu\in\bbr$. Then $\wh{h}=f$ and hence, due to the ``inversion formula'' 
 (see \cite[p. 158]{Katznelson2004}) and Theorem \ref{Theorem 1.36},
 \begin{align}
  \wh{f\Phi_\la}
   &= \wh{\Phi_\la\wh{h}}\overset{(\scriptsize\ref{Equation 1.12})}{\underset{(\scriptsize\ref{Equation 1.19})}{=}}\wh{\wh{F_\la*h}}
    = (F_\la*h)(-\hspace{0.5mm}\cdot\hspace{0.5mm})=F_\la*g\stackrel{L^\delta(\bbc)}{\to} g\quad\text{as }\la\to\infty \nonumber
 \end{align}
 for any $1\leq\delta<\al$. Thus, applying again Theorem \ref{Theorem 6.7}, we deduce
 \begin{align}
  \int_{-\infty}^\infty f\Phi_\la\,dM 
   &\stackrel{(\scriptsize\ref{Equation 6.6})}{=}\frac{1}{\sqrt{2\pi}}\int_{-\infty}^\infty\wh{f\Phi_\la}\,dL = \frac{1}{\sqrt{2\pi}}\int_{-\infty}^\infty (F_\la *g)\,dL \limp\frac{1}{\sqrt{2\pi}}\int_{-\infty}^\infty g\,dL 
    = L_b - L_a\quad\text{as }\la\to\infty \nonumber
 \end{align}
 and the claimed spectral representation for regularly varying L\'evy processes is shown.  \\
\end{proof}
\begin{Remark}$~~$
 \ben[(i)]
  \item
   If the L\'evy process is even symmetric $\al$-stable with $\al\in(1,2)$, then the limit in probability which occurs in the spectral representation of the L\'evy process in Theorem \ref{Theorem 6.4} 
   can be replaced by a limit in $L^p(\Om, \scrf, \bbP; \bbc^d)$ for any $p<\al$. For more details and a comprehensive treatment of $\al$-stable concepts we refer to \cite{Samorodnitskyetal1994}. 
  \item
   The assumption (\ref{Equation 4.10}), which has been used in the second step in order to extend integration with respect to $M$ to more general integrands, is strong. 
   However, as one can observe from Step 3 in the preceding proof, it holds for any continuous function $f$ with compact support on $\bbr$ if $f$ is in addition of bounded variation. All the functions appearing in
   connection with multivariate regularly varying CARMA processes shall be of this type (cf. upcoming Proposition \ref{Proposition 6.9} and Lemma \ref{Lemma 4.5}).
 \een
\label{Remark 4.11}
\end{Remark}
\subsection{Properties of the Resulting Random Content}
Now we discuss distributional and moment properties of the resulting random content $M$ defined in Theorem \ref{Theorem 6.4}. Moreover, we study characteristics of its increments 
and its corresponding L\'evy measure. Regarding the associated noise $M$ there is, apart from a brief treatment of the Fourier transform of a compensated Poisson process in \cite[Example 8.4]{Daleyetal2003}, 
to the best of our knowledge hardly anything in the literature.

Let us first characterize the distribution of the random content $M$ by applying its definition and \cite[Theorem 2.7]{Rajputetal1989}.
\begin{Proposition}
 Let the driving L\'evy process $L$ have generating triplet $(\ga,\Sigma,\nu)$, then the distribution of $M(A)$ is infinitely divisible for any $A\in\scre(\bbr)$ and its generating triplet is 
 $(\ga_{M(A)},\Sigma_{M(A)},\nu_{M(A)})$, where
 \[\ga_{M(A)} = \frac{1}{\sqrt{2\pi}}\int_\bbr\left(\wh{\mathds{1}_A}(s)\ga + \int_{\bbr^d} \wh{\mathds{1}_A}(s)x
  \left(\mathds{1}_{[0,\sqrt{2\pi}]}(\|\wh{\mathds{1}_A}(s)x\|)-\mathds{1}_{[0,1]}(\left\|x\right\|)\right)\nu(dx)\right)ds,\]
 \[\Sigma_{M(A)} = \frac{1}{2\pi}\int_\bbr \wh{\mathds{1}_A}(s)\Sigma\,\wh{\mathds{1}_A}(s)^* ds\quad\text{and}\quad
  \nu_{M(A)}(B)=\int_\bbr\int_{\bbr^d}\mathds{1}_B\left(\frac{1}{\sqrt{2\pi}}\wh{\mathds{1}_A}(s)x\right)\,\nu(dx)\,ds,\ B\in\scrb(\bbc^d_*).\]
 \label{Proposition 6.18}
\end{Proposition}
The next question we cover is whether the increments of $M$ are independently or stationarily scattered.
\begin{Theorem}
 Let $L=(L_t)_{t\in\bbr}$ be a two-sided L\'evy process in $\bbr^d$ with $\bbE[L_1]=0$ and generating triplet $(\ga,\Sigma,\nu)$. Assume moreover that $\nu$ is regularly varying with index 
 $\al\in(1,2]$ and let $M$ be the associated regularly varying random content of Theorem \ref{Theorem 6.4}. Then $M$ is neither independently nor stationarily scattered.
 \label{Theorem 6.14}
\end{Theorem}
\begin{proof}[\textbf{Proof.}]
 Assume that the increments of $M=\big(M^{(1)},\ldots,\,M^{(d)}\big)$ were independent such that, in particular, $\mathrm{Re}\,M^{(1)}([a_1,b_1))$ and $\mathrm{Re}\,M^{(1)}([a_2,b_2))$ have to be 
 independent for disjoint intervals $[a_1,b_1)$ and $[a_2,b_2)$. Since $\nu$ is regularly varying, it is by definition non-trivial and thus w.l.o.g. the L\'evy measure of the first component $L^{(1)}$ of $L$, denoted by
 $\nu^{(1)}$, has to be non-trivial. The L\'evy measure of 
 \[\begin{pmatrix}\mathrm{Re}\,M^{(1)}([a_1,b_1)) \\ \mathrm{Re}\,M^{(1)}([a_2,b_2)) \end{pmatrix}
  = \frac{1}{\sqrt{2\pi}}\int_\bbr\begin{pmatrix}\mathrm{Re}\,\wh{\mathds{1}_{[a_1,b_1)}} \\ \mathrm{Re}\,\wh{\mathds{1}_{[a_2,b_2)}}\end{pmatrix}dL^{(1)} =: \int_\bbr g\,dL^{(1)}\]
 is given by (see again \cite[Theorem 2.7]{Rajputetal1989}) $\widetilde{\nu}(B)=\int_\bbr\int_{\bbr}\mathds{1}_B(g(s)x)\,\nu^{(1)}(dx)\,ds$ for any $B\in\scrb(\bbr^2_*).$ Then, under our 
 independence assumption, \cite[Exercise 12.8]{Sato1999} implies that $\widetilde{\nu}$ has to be concentrated on $D:=\left\{(x,y)\in\bbr^2:\,x=0\text{ or }y=0\right\}$. 
 
 But, letting $N:=\left\{s\in\bbr:\,\sin\left(b_1s\right)=\sin\left(a_1s\right)\right\}\cup\left\{s\in\bbr:\,\sin\left(b_2s\right)=\sin\left(a_2s\right)\right\}$, we have
 \begin{align}
  \wt{\nu}\left(\bbr^2\backslash D\right)&=\int_\bbr\int_\bbr\mathds{1}_{\bbr^2\backslash D}(g(s)x)\,\nu^{(1)}(dx)\,ds
   \geq\la^1\otimes\nu^{(1)}\big((\bbr\backslash N)\times(\bbr\backslash\left\{0\right\})\big) = \infty, \nonumber
 \end{align}
 since $N$ is a Lebesgue null set. This obviously gives a contradiction and hence the increments of $M$ cannot be independent.

 We still have to show that the increments are not stationarily scattered either. On the contrary if
 \[\left\{M(A):\,A\in\scre(\bbr)\right\}\eqd\left\{M_\tau(A)=M(A+\tau):\,A\in\scre(\bbr)\right\}\]
 for all $\tau\in\bbr$, then 
 \begin{align}
  &\left\{\int_{-\la}^\la \frac{e^{it\mu}-1}{i\mu}\cdot\left(1-\frac{\left|\mu\right|}{\la}\right) M(d\mu):\,\la>0,\,t\in\bbr\right\} \nonumber \\
  &\qquad\qquad\qquad\eqd\left\{\int_{-\la}^\la \frac{e^{it\mu}-1}{i\mu}\cdot\left(1-\frac{\left|\mu\right|}{\la}\right) M_\tau(d\mu):\,\la>0,\,t\in\bbr\right\}. \nonumber
 \end{align}
 With $f_t(\mu):=\frac{e^{it\mu}-1}{i\mu},\,g_t(\mu)=\sqrt{2\pi}\mathds{1}_{[0,t)}(\mu)$ and the same notations as in the proof of Theorem \ref{Theorem 6.4}, we have
 \begin{align}
  \int_{-\infty}^\infty f_t\Phi_\la(\mu)\,M_\tau(d\mu)&=\int_{-\infty}^\infty f_t\Phi_\la(\mu-\tau)\,M(d\mu) = \frac{1}{\sqrt{2\pi}}\int_{-\infty}^\infty e^{-i\tau\xi}\wh{f_t\Phi_\la}(\xi)\,L(d\xi) \nonumber \\
   &=\frac{1}{\sqrt{2\pi}}\int_{-\infty}^\infty e^{-i\tau\xi}\cdot\left(F_\la*g_t\right)(\xi)\,L(d\xi) \limp \int_{-\infty}^\infty e^{-i\tau\xi}\mathds{1}_{[0,t)}(\xi)\,L(d\xi) \nonumber
 \end{align}
 as $\la\to\infty.$ Thus, for all $\tau\in\bbr$, $\left\{L_t:\,t\in\bbr\right\}\eqd\left\{\int_{-\infty}^\infty e^{-i\tau\xi}\mathds{1}_{[0,t)}(\xi)\,L(d\xi):\,t\in\bbr\right\}.$
 In particular we obtain $L_1\eqd\int_{-\infty}^\infty e^{-i\pi\xi}\mathds{1}_{[0,1)}(\xi)\,L(d\xi)$ and hence 
 \begin{align}
  \nu\big(\bbr^d\backslash\left\{0\right\}\big)&=\int_\bbr\int_{\bbr^d}\mathds{1}_{\bbr^d\backslash\left\{0\right\}}\left(e^{-i\pi s}\mathds{1}_{[0,1)}(s)x\right)\nu(dx)\,ds 
   = \la^1\otimes\nu\big(\left\{0\right\}\times\bbr^d\backslash\left\{0\right\}\big)=0, \nonumber
 \end{align}
 a contradiction, since $\nu$ was supposed to be regularly varying and thus by definition non-trivial.
\end{proof}
\begin{Remark}$~~$
 \begin{enumerate}[(i)]
  \item
   Our proof shows in addition that the corresponding random measures of arbitrary disjoint half-open bounded intervals are always dependent. With a slight modification 
   one can even show that the same is true for arbitrary disjoint elementary sets.
  \item
   Note that the proof is still correct if we replace the assumption that $L$ is regularly varying by $\bbE[\left\|L_1\right\|^2]<\infty$ and $\nu(\bbr^d)\neq 0$. In this case the underlying L\'evy process 
   has finite second moments and a bona fide spectral representation has been derived in \cite{Marquardtetal2007}. It is well-known that the corresponding random noise is then defined for all bounded Borel 
   sets and one obtains a so-called random orthogonal measure (see \cite{Doob1953, Rozanov1967} for comprehensive treatments) with uncorrelated but dependent increments.
  \item
   The assumption that $\nu$ is regularly varying is not explicitly used in the proof, it however assures that $\nu\not\equiv 0$. If this non-triviality is not guaranteed, the result of the proposition 
   becomes incorrect since it is well-known that the corresponding random noise in the standard Brownian case has orthogonal and stationary increments (see e.g. \cite[Section 2.1, Lemma 5]{Arato1982}).
 \end{enumerate}
 \label{Remark 6.15} 
\end{Remark}
In addition the following properties about moments and the local behavior of the L\'evy measure of the random content at zero can be shown. For a definition of the notion of $\delta$-variation see for instance
\cite{Bretagnolle1972}.
\begin{Proposition}
Let $L=(L_t)_{t\in\bbr}$ be a two-sided L\'evy process in $\bbr^d$ with $\bbE[L_1]=0$ and generating triplet $(\ga,\Sigma,\nu)$. Assume moreover that $\nu\in RV(\al, l, \kappa_\nu)$ for some $\al\in(1,2]$ 
and let $M$ be the associated regularly varying random content of Theorem \ref{Theorem 6.4}. Then the process
\[Z_t:=M\left(\left[\left.0,t\right)\right.\right)=\int_{-\infty}^\infty\frac{1-e^{-it\mu}}{2\pi i\mu}\,L(d\mu),\quad t\in\bbr_+,\]
is regularly varying with index $\al$. 

Furthermore we have, for any $t\in\bbr_+$, the following results for $Z_t$ and its corresponding L\'evy measure $\nu_{Z_t}$:
\ben[(i)]
 \item
  $\int_{\left\{\left\|x\right\|\leq 1\right\}}\left\|x\right\|\nu_{Z_t}(dx)=\infty$ and thus $\nu_{Z_t}$ is in particular infinite.
 \item
  $\bbE[\left\|Z_t\right\|^p]<\infty$ for any $0<p<\al$ and $\bbE[\left\|Z_t\right\|^p]=\infty$ for any $p>\al$.
 \item
  For any $\delta\in(1,\al)$ the integral $\int_{\left\{\left\|x\right\|\leq 1\right\}}\left\|x\right\|^\delta\nu(dx)$ is finite iff $\int_{\left\{\left\|x\right\|\leq 1\right\}}\left\|x\right\|^\delta\nu_{Z_t}(dx)$ is finite, 
  i.e. the L\'evy process $L$ has a.s. finite $\delta$-variation if and only if $\Sigma=0$ and $\int_{\left\{\left\|x\right\|\leq 1\right\}}\left\|x\right\|^\delta\nu_{Z_t}(dx)<\infty$. 
  Moreover, if the L\'evy process satisfies in addition $\bbE[\left\|L_1\right\|^\al]<\infty$, then the statement is also true for $\delta=\al$.
 \item
  The implication 
  \[\int_{\left\{\left\|x\right\|\leq 1\right\}}\left\|x\right\|^\delta\nu_{Z_t}(dx)<\infty \Rightarrow \int_{\left\{\left\|x\right\|\leq 1\right\}}\left\|x\right\|^\delta\nu(dx)<\infty\]
  is valid for every $\delta\in(1,2)$.
\een
\label{Proposition 6.13}
\end{Proposition}
\begin{proof}[\textbf{Proof.}]
 We first show that $(Z_t)_{t\in\bbr_+}$ is regularly varying with index $\al$, i.e. all finite dimensional margins are regularly varying with index $\al$. Let $(t_1,\ldots,t_m)^\prime\in\bbr_+^m$ and observe that 
 \begin{align}
  \begin{pmatrix} Z_{t_1} \\ \vdots \\ Z_{t_m} \end{pmatrix}
   &=\frac{1}{\sqrt{2\pi}}
    \begin{pmatrix} \int_{-\infty}^\infty \wh{\mathds{1}_{[0,t_1)}}(\mu)\,L(d\mu) \\ \vdots \\ \int_{-\infty}^\infty \wh{\mathds{1}_{[0,t_m)}}(\mu)\,L(d\mu) \end{pmatrix}
    = \frac{1}{\sqrt{2\pi}}\int_{-\infty}^\infty g_{t_1,\ldots,t_m}(\mu)\,L(d\mu) \nonumber
 \end{align}
 where $g_{t_1,\ldots,t_m}:\bbr\to M_{md\times d}(\bbc)$ is defined by 
 \[g_{t_1,\ldots,t_m}(\mu):= \begin{pmatrix} \wh{\mathds{1}_{[0,t_1)}}(\mu)\mathrm{I}_d \\ \vdots \\ \wh{\mathds{1}_{[0,t_m)}}(\mu)\mathrm{I}_d \end{pmatrix}.\]
 Since we obviously have $g_{t_1,\ldots,t_m}\in L^\al(M_{md\times d}(\bbc))$ and $\kappa_\nu\left(g_{t_1,\ldots,t_m}^{-1}(\mu)(\bbc^{md}\backslash\left\{0\right\})\cap\bbr^d\right)=0$ does not hold 
 for almost every $\mu$, a simple application of \cite[Theorem 3.2]{Moseretal2010} shows that the process $(Z_t)_{t\in\bbr_+}$ is regularly varying of index $\al$. This also implies (ii).
 
 Recall that the L\'evy measure of $Z_t$, identifying again $\bbc$ with $\bbr^2$ and $\bbc^d$ with $\left(\bbr^2\right)^d$, is given by
 \[\nu_{Z_t}(A)=\int_\bbr\int_{\bbr^d}\mathds{1}_A\left(\frac{1-e^{-it\mu}}{2\pi i\mu}x\right)\nu(dx)\,d\mu = f(\la^1\otimes\nu)(A),\quad A\in\scrb\big(\bbc^d_*\big),\]
 where $f:\bbr\times\bbr^d\to\bbc^d,\,f(\mu,x):=\frac{1-e^{-it\mu}}{2\pi i\mu}x$ and $f(\la^1\otimes\nu)$ denotes the image measure of $\la^1\otimes\nu$ by $f$ on $(\bbc^d, \scrb(\bbc^d))$. 

 Now, for any $\delta\in[1,2)$, we observe
 \begin{align}
  \int_{\left\{\left\|x\right\|\leq 1\right\}}
   &\left\|x\right\|^\delta\nu_{Z_t}(dx) = \int_{\left\{(\mu,y)\in\bbr\times\bbr^d:\ \frac{\sqrt{1-\cos{(t\mu)}}}{\left|\mu\right|}\left\|y\right\|\leq\sqrt{2}\pi\right\}}
    \left(\frac{\left\|y\right\|}{\sqrt{2}\pi}\right)^\delta\frac{(1-\cos{(t\mu)})^{\frac{\delta}{2}}}{\left|\mu\right|^\delta}\,(\la^1\otimes\nu)(d(\mu,y)) \nonumber \\
   &=\left(\frac{1}{\sqrt{2}\pi}\right)^\delta\cdot\int_{\bbr^d\backslash\left\{0\right\}}\left\|y\right\|^\delta\left(\int_{\left\{\mu\in\bbr:\ \frac{\sqrt{1-\cos{(t\mu)}}}
    {\left|\mu\right|}\left\|y\right\|\leq\sqrt{2}\pi\right\}}\frac{(1-\cos{(t\mu)})^{\frac{\delta}{2}}}{\left|\mu\right|^\delta}\,d\mu\right)\nu(dy) \label{Equation 2.4}
 \end{align}
 due to Fubini's Theorem. For $\delta=1$ the inner integral in (\ref{Equation 2.4}) is infinite for all $y\in\bbr^d\backslash\left\{0\right\}$ and we deduce
 $\int_{\left\{\left\|x\right\|\leq 1\right\}}\left\|x\right\|\nu_{Z_t}(dx)=\infty.$ Thus (i) is shown.
  
 Let now $\delta\in(1,\al)$ and assume $\int_{\left\{\left\|x\right\|\leq 1\right\}}\left\|x\right\|^\delta\nu(dx)<\infty$. Note that the inner integral in (\ref{Equation 2.4}) can be bounded by
 $\int_\bbr\frac{(1-\cos{(t\mu)})^{\frac{\delta}{2}}}{\left|\mu\right|^\delta}\,d\mu=:C(\delta)<\infty$ and thus (\ref{Equation 2.4}) becomes
 \begin{align}
  \int_{\left\{\left\|x\right\|\leq 1\right\}}&\left\|x\right\|^\delta\nu_{Z_t}(dx)\leq C(\delta)\cdot\left(\frac{1}{\sqrt{2}\pi}\right)^\delta\cdot\left(\int_{\left\{\left\|y\right\|\leq 1\right\}}
   \left\|y\right\|^\delta\nu(dy) + \int_{\left\{\left\|y\right\|>1\right\}}\left\|y\right\|^\delta\nu(dy)\right). \nonumber
 \end{align}
 The first integral on the right-hand side is finite by assumption and the second integral is finite as well since $1<\delta<\al$ and hence the underlying L\'evy process has a finite $\delta$-th moment 
 (cf. \cite[Corollary 25.8]{Sato1999}). Hence, the integral $\int_{\left\{\left\|x\right\|\leq 1\right\}}\left\|x\right\|^\delta\nu_{Z_t}(dx)$ is finite.

 If in addition $\bbE[\left\|L_1\right\|^\al]<\infty$ holds, then it is obvious that finiteness of $\int_{\left\{\left\|x\right\|\leq 1\right\}}\left\|x\right\|^\al\nu(dx)$ still implies that 
 $\int_{\left\{\left\|x\right\|\leq 1\right\}}\left\|x\right\|^\al\nu_{Z_t}(dx)$ is also finite. 
 
 Conversely, let $\delta\in(1,2)$ and assume that $\int_{\left\{\left\|x\right\|\leq 1\right\}}\left\|x\right\|^\delta\nu_{Z_t}(dx)<\infty$. Then (cf. (\ref{Equation 2.4})) 
 \begin{align}
  \int_{\left\{\left\|x\right\|\leq 1\right\}}&\left\|x\right\|^\delta\nu_{Z_t}(dx) = \left(\frac{1}{\sqrt{2}\pi}\right)^\delta\int_\bbr\frac{(1-\cos{(t\mu)})^{\frac{\delta}{2}}}{\left|\mu\right|^\delta}
   \int_{\left\{y\in\bbr^d\backslash\left\{0\right\}:\ \frac{\sqrt{1-\cos{(t\mu)}}}{\left|\mu\right|}\left\|y\right\|\leq\sqrt{2}\pi\right\}}\left\|y\right\|^\delta\nu(dy)\,d\mu \nonumber
 \end{align}
 and due to $\left\{y\in\bbr^d\backslash\left\{0\right\}:\ \left\|y\right\|\leq\pi\cdot\left|\mu\right|\right\}\subseteq\Big\{y\in\bbr^d\backslash\left\{0\right\}:\ \frac{\sqrt{1-\cos{(t\mu)}}}{\left|\mu\right|}
 \left\|y\right\|\leq\sqrt{2}\pi\Big\}$ for all $\mu\neq 0$, we deduce
 \begin{align}
  \int_{\left\{\left\|x\right\|\leq 1\right\}}\left\|x\right\|^\delta\nu_{Z_t}(dx)
  &\geq\left(\frac{1}{\sqrt{2}\pi}\right)^\delta\cdot\int_{\left\{\left|\mu\right|\geq 1\right\}}\frac{(1-\cos{(t\mu)})^{\frac{\delta}{2}}}{\left|\mu\right|^\delta}
   \left(\int_{\left\{\left\|y\right\|\leq\pi\left|\mu\right|\right\}}\left\|y\right\|^\delta\nu(dy)\right)d\mu \nonumber \\
  &\geq\left(\frac{1}{\sqrt{2}\pi}\right)^\delta\cdot\int_{\left\{\left|\mu\right|\geq 1\right\}}\frac{(1-\cos{(t\mu)})^{\frac{\delta}{2}}}{\left|\mu\right|^\delta}\,d\mu\cdot
   \int_{\left\{\left\|y\right\|\leq 1\right\}}\left\|y\right\|^\delta\nu(dy). \nonumber
 \end{align}
 The first integral on the right-hand side is strictly positive and finite since $\delta>1$. Hence we obtain $\int_{\left\{\left\|y\right\|\leq 1\right\}}\left\|y\right\|^\delta\nu(dy)<\infty$ and (iv) is shown.
 
 It is well-known that, for any $\delta\in(1,2)$, $L$ has a.s. finite $\delta$-variation if and only if $\Sigma=0$ and $\int_{\left\{\left\|x\right\|\leq 1\right\}}\left\|x\right\|^\delta\nu(dx)<\infty$
 (cf. \cite[Theorem IIIb]{Bretagnolle1972}). This yields the additional statement of (iii) and completes the proof of the proposition.
\end{proof}
\begin{Remark}
 If $L$ is assumed to be symmetric $\al$-stable with $\al\in(1,2)$, then $(Z_t)_{t\in\bbr_+}$ becomes itself a symmetric $\al$-stable stochastic process. In this case one can again show that, for any $\delta\in(0,2)$
 and $t\in\bbr_+$, the integral $\int_{\left\{\left\|x\right\|\leq 1\right\}}\left\|x\right\|^\delta\nu(dx)$ is finite if and only if the integral $\int_{\left\{\left\|x\right\|\leq 1\right\}}\left\|x\right\|^\delta\nu_{Z_t}(dx)$
 is finite which is in turn the case if and only if $\delta>\al$ (cf. \cite[Theorem 14.3]{Sato1999}).
\label{Remark 6.16}
\end{Remark}
We conclude this section by the following version of Proposition \ref{Proposition 6.13} for the case where the underlying L\'evy process has finite second moments.
\begin{Proposition}
 Let $L=(L_t)_{t\in\bbr}$ be a two-sided square-integrable L\'evy process in $\bbr^d$ such that $\bbE[L_1]=0$ and let $M$ be the corresponding random orthogonal measure of Theorem \ref{Theorem 6.4} 
 (cf. Remark \ref{Remark 6.15}). Moreover, letting $(\ga,\Sigma,\nu)$ the generating triplet of $L$, we assume that $\nu\not\equiv 0$. Then, for all $t\in\bbr_+$, we have the following results for 
 $Z_t:=M\left(\left[\left.0,t\right)\right.\right)=\int_{-\infty}^\infty\frac{1-e^{-it\mu}}{2\pi i\mu}\,L(d\mu)$:
 \ben[(i)]
  \item
   $\int_{\left\{\left\|x\right\|\leq 1\right\}}\left\|x\right\|\nu_{Z_t}(dx)=\infty$ and thus $\nu_{Z_t}$ is in particular infinite.
  \item
   For any $\delta\in(1,2)$, the L\'evy process $L$ has a.s. finite $\delta$-variation if and only if $\Sigma=0$ and $\int_{\left\{\left\|x\right\|\leq 1\right\}}\left\|x\right\|^\delta\nu_{Z_t}(dx)<\infty$.
  \item
   For any $\beta>0$, we have $\bbE[\left\|L_1\right\|^\beta]<\infty$ if and only if $\bbE[\left\|Z_t\right\|^\beta]<\infty$.
  \item
   If $\bbE[\exp\left\{\beta\left\|L_1\right\|\right\}]<\infty$ for some $\beta>0$, then $\bbE[\exp\left\{\eta(t)\left\|Z_t\right\|\right\}]<\infty$ with 
   \[\eta(t)=\frac{\pi}{\sqrt{2}c(t)}\beta\qquad\text{and}\qquad c(t)=\sup\limits_{\mu\in\bbr}\frac{\sqrt{1-\cos(t\mu)}}{\left|\mu\right|}\in(0,\infty).\]
 \een
\label{Proposition 6.19}
\end{Proposition}
\begin{proof}[\textbf{Proof.}]
 (i), (ii) and (iii) can be shown analogously to the proof of Proposition \ref{Proposition 6.13}. As to (iv), we know by virtue of \cite[Corollary 25.8]{Sato1999} that $\bbE[\exp{\left\{\eta\left\|Z_t\right\|\right\}}]<\infty$ 
 if and only if $\int_{\left\{\left\|x\right\|>1\right\}}\exp{\left\{\eta\left\|x\right\|\right\}}\,\nu_{Z_t}(dx)$ is finite. As in (\ref{Equation 2.4}), we have
 \begin{align}
  \int_{\left\{\left\|x\right\|>1\right\}}&\exp{\left\{\eta\left\|x\right\|\right\}}\,\nu_{Z_t}(dx) \nonumber \\
   &=\int_{\bbr^d\backslash\left\{0\right\}}\int_{\left\{\mu\in\bbr:\ \frac{\sqrt{1-\cos{(t\mu)}}}{\left|\mu\right|}\left\|y\right\|>\sqrt{2}\pi\right\}}
    \exp{\left\{\eta\frac{\left\|y\right\|}{\sqrt{2}\pi}\frac{\sqrt{1-\cos{(t\mu)}}}{\left|\mu\right|}\right\}}\,d\mu\,\nu(dy). \label{Equation 2.6}
 \end{align}
 Setting $c(t):=\sup_{\mu\in\bbr}\frac{\sqrt{1-\cos{(t\mu)}}}{\left|\mu\right|}$, observe first that $c(t)\in(0,\infty)$ for any $t>0$. Thus we obtain for any 
 $y\in\bbr^d\backslash\left\{0\right\}$ with $\left\|y\right\|\leq\frac{\sqrt{2}\pi}{c(t)}$ that $\frac{\sqrt{1-\cos{(t\mu)}}}{\left|\mu\right|}\left\|y\right\|\leq\sqrt{2}\pi$ for all $\mu\in\bbr$. This gives that 
 the inner integral in (\ref{Equation 2.6}) vanishes for all $y\in\bbr^d\backslash\left\{0\right\}$ with $\left\|y\right\|\leq\frac{\sqrt{2}\pi}{c(t)}$ and hence, using in addition the relation 
 $\Big\{\mu\in\bbr:\ \frac{\sqrt{1-\cos{(\mu t)}}}{\left|\mu\right|}\left\|y\right\|>\sqrt{2}\pi\Big\}\subseteq\left\{\mu\in\bbr:\ \left|\mu\right|<\frac{\left\|y\right\|}{\pi}\right\}$,
 (\ref{Equation 2.6}) can be bounded by
 \begin{align}
  \int_{\left\{\left\|x\right\|>1\right\}}\exp{\left\{\eta\left\|x\right\|\right\}}\,\nu_{Z_t}(dx)
  &\leq\int_{\left\{\left\|y\right\|>\frac{\sqrt{2}\pi}{c(t)}\right\}}\int_{\left\{\left|\mu\right|<\frac{\left\|y\right\|}{\pi}\right\}}
   \exp{\left\{\eta\frac{\left\|y\right\|}{\sqrt{2}\pi}\frac{\sqrt{1-\cos{(t\mu)}}}{\left|\mu\right|}\right\}}\,d\mu\,\nu(dy) \nonumber \\
  &\leq\frac{2}{\pi}\cdot\int_{\left\{\left\|y\right\|>\frac{\sqrt{2}\pi}{c(t)}\right\}}\left\|y\right\|\cdot\exp{\left\{\eta\frac{c(t)\cdot\left\|y\right\|}{\sqrt{2}\pi}\right\}}\,\nu(dy). \label{Equation 2.7}
 \end{align}
 Since 
 $\bbE\big[\left\|L_1\right\|\cdot\exp{\big\{\frac{\beta}{2}\left\|L_1\right\|\big\}}\big]\leq\bbE[\left\|L_1\right\|^2]^{\frac{1}{2}}\cdot\bbE\left[\exp{\left\{\beta\left\|L_1\right\|\right\}}\right]^{\frac{1}{2}}<\infty$
 by assumption, \cite[Corollary 25.8]{Sato1999} shows that the right-hand side of (\ref{Equation 2.7}) is finite for $\eta=\eta(t):=\frac{\pi}{\sqrt{2}c(t)}\beta>0$.
\end{proof}
So the random content inherits moments exactly and always has a L\'evy measure around zero with infinite activity and $\int_{\left\{\left\|x\right\|\leq 1\right\}}\left\|x\right\|\nu_{Z_t}(dx)=\infty$. 
Moreover, if $L$ has a moment generating function in a neighborhood of zero then so does $Z_t$. 
\section{Moving Averages and Regularly Varying MCARMA Processes}
Let us now turn to multivariate CARMA processes which are driven by regularly varying L\'evy processes. Before turning to their definition we establish one proposition and one lemma. The following result 
gives further insight into the spectral representation of moving averages of regularly varying L\'evy processes.
\begin{Proposition}
 Let $L=(L_t)_{t\in\bbr}$ be a L\'evy process in $\bbr^d$ with generating triplet $(\ga,\Sigma,\nu)$ where $\nu\in RV(\al, l, \kappa_\nu)$ with $\al\in(1,2]$ and suppose $\bbE[L_1]=0$. 
 Let $M$ be the corresponding random content of Theorem \ref{Theorem 6.4} and assume that $h\in L^1(M_d(\bbc))\cap L^\al(M_d(\bbc))$ such that in addition $h$ is bounded and its 
 Fourier transform $\wh{h}$ is of bounded variation on compacts. Define 
 \[G_t:=\mathop{\bbP-\lim}\limits_{\la\to\infty}\int_{-\la}^\la e^{it\mu}\,\wh{h}(\mu)\left(1-\frac{\left|\mu\right|}{\la}\right) M(d\mu), \quad t\in\bbr.\]
 Then, for all $t\in\bbr$,
 \[G_t = \frac{1}{\sqrt{2\pi}}\int_{-\infty}^\infty h(t-\mu)\,L(d\mu).\]
 If in addition $\kappa_\nu\left(h^{-1}(s)(\bbc^d\backslash\left\{0\right\})\cap\bbr^d\right)=0$ does not hold for almost every $s$, then the process $(G_t)_{t\in\bbr}$ is also regularly varying of index $\al$.
\label{Proposition 6.9}
\end{Proposition}
\begin{proof}[\textbf{Proof.}]
 Since $h\in L^1(M_d(\bbc))$, the Fourier transform $\wh{h}$ is obviously continuous and thus the function 
 \[f_{\la,t}(\mu):=e^{it\mu}\,\wh{h}(\mu)\left(1-\frac{\left|\mu\right|}{\la}\right)\mathds{1}_{[-\la,\la]}(\mu),\quad\mu\in\bbr,\]
 is continuous with compact support on $\bbr$ and has bounded variation by assumption. Consequently, it can be approximated in the $L^1(M_d(\bbc))$-norm by a sequence of $\scre(\bbr)$-simple functions $f_n$
 satisfying in addition $\wh{f_n}\to\wh{f_{\la,t}}$ in $L^\delta(M_d(\bbc))$ as $\nto$ for any $\delta\in(1,\al)$ (this can be shown in the same way as in Step 3 of the proof of Theorem \ref{Theorem 6.4}, see also
 Remark \ref{Remark 4.11}). Thus $f_{\la,t}$ is $M$-integrable for any $\la>0$ and $t\in\bbr$. 
 
 Then
 \[G_t=\mathop{\bbP-\lim}\limits_{\la\to\infty}\int_{-\infty}^\infty f_{\la,t}(\mu)\,M(d\mu)\stackrel{(\scriptsize\ref{Equation 6.6})}{=}\mathop{\bbP-\lim}\limits_{\la\to\infty}\frac{1}
  {\sqrt{2\pi}}\int_{-\infty}^\infty\wh{f_{\la,t}}(\mu)\,L(d\mu).\]
 Setting $\Phi_\la(\mu):=\big(1-\frac{\left|\mu\right|}{\la}\big)\mathds{1}_{[-\la,\la]}(\mu)$ and $h_t(\mu):= h(\mu + t)$, we have 
 \[f_{\la,t} = \Phi_\la e^{it\hspace{0.5mm}\cdot\hspace{0.5mm}} \wh{h} =\Phi_\la\wh{h_t}\overset{(\scriptsize\ref{Equation 1.12})}{\underset{(\scriptsize\ref{Equation 1.19})}{=}}\wh{F_\la*h_t}\]
 and thus, due to the ``inversion formula'' (see \cite[p. 158]{Katznelson2004}) and Theorem \ref{Theorem 1.36},
 \[\wh{f_{\la,t}}=\wh{\wh{F_\la * h_t}}=(F_\la * h_t)(-\hspace{0.5mm}\cdot\hspace{0.5mm}) = F_\la * (h_t(-\hspace{0.5mm}\cdot\hspace{0.5mm}))
  \stackrel{L^\delta(M_d(\bbc))}{\to}h_t(-\hspace{0.5mm}\cdot\hspace{0.5mm})\quad\text{as }\la\to\infty\]
 for any $1\leq\delta<\al$ since $h_t\in L^1(M_d(\bbc))\cap L^\al(M_d(\bbc))$. Hence Theorem \ref{Theorem 6.7} yields
 \[G_t=\mathop{\bbP-\lim}\limits_{\la\to\infty}\frac{1}{\sqrt{2\pi}}\int_{-\infty}^\infty F_\la*(h_t(-\hspace{0.5mm}\cdot\hspace{0.5mm}))\,dL = \frac{1}{\sqrt{2\pi}}\int_{-\infty}^\infty 
  h_t(-\hspace{0.5mm}\cdot\hspace{0.5mm})\,dL\]
 for all $t\in\bbr$. 
 
 The additional statement follows from \cite[Corollary 3.5]{Moseretal2010}.
\end{proof}
The next lemma verifies the assumptions of Proposition \ref{Proposition 6.9} for regularly varying MCARMA processes.
\begin{Lemma} 
 Let $p,\,q\in\bbn_0$ with $p>q$ and $A_1,\ldots,\,A_p,\,B_0,\,B_1,\ldots,\,B_q\in M_d(\bbc)$ with $B_0\neq 0$. 
 Define
 \[P:\bbc\to M_d(\bbc),\,z\mapsto z^p\mathrm{I}_d + z^{p-1}A_1+\ldots+A_p,\]
 \[Q:\bbc\to M_d(\bbc),\,z\mapsto z^qB_0 + z^{q-1}B_1+\ldots+B_q\]
 and assume that $\scrn(P):=\left\{z\in\bbc:\,\mathrm{det}(P(z))=0\right\}\subseteq\bbr\backslash\left\{0\right\} + i\bbr$. Then the function 
 \[g:\bbr\to M_d(\bbc),\ g(\mu):=P(i\mu)^{-1}Q(i\mu)\]
 is continuous and of bounded variation on compacts. Moreover $g = \wh{h}$ in $L^2(M_d(\bbc))$ where, for all $\mu\neq 0$,
 \begin{align}
  h(\mu)
   &:=\frac{1}{\sqrt{2\pi}}\int_\bbr e^{i\mu s}P(is)^{-1}Q(is)\,ds = \sqrt{2\pi}\sum\limits_\la\sum\limits_{s=0}^{m(\la)-1}\left(\mu^se^{\la\mu}
    \mathds{1}_{\left\{\mathrm{Re}(\la)\cdot\mu<0\right\}}C_{\la s}\right)\in L^\delta(M_d(\bbc))\label{Equation 4.2}
 \end{align}
 for any $\delta\geq 1$. Here $\sum\limits_\la$ denotes the sum over all distinct zeros in $\scrn(P)$, the multiplicity of the zero $\la$ is written as $m(\la)$ and $C_{\la s}$ are constant 
 complex-valued $d\times d$ matrices.
\label{Lemma 4.5}
\end{Lemma}
\begin{proof}[\textbf{Proof.}]
 We need the following consequence of the residue theorem from complex analysis (see for instance \cite[Section VI.2, Theorem 2.2]{Lang1993} or \cite[Section III.7, Theorem 7.11]{Freitagetal2000}):
 let $p,\,q:\bbc\to\bbc$ be polynomials where $p$ is of higher degree than $q$. Assume that $p$ has no zeros on the real line. Then 
 \[\int_{-\infty}^\infty e^{i\mu t}\frac{q(t)}{p(t)}\,dt = \left\{\begin{array}{l}2\pi i\cdot\sum\limits_{\substack{z\in\bbc:\,\mathrm{Im}(z)>0,\\ p(z)=0}}\mathrm{Res}_zf,\quad\hfill\mu>0 \\
  -2\pi i\cdot\sum\limits_{\substack{z\in\bbc:\,\mathrm{Im}(z)<0,\\ p(z)=0}}\mathrm{Res}_zf,\quad\hfill\mu<0\end{array}\right\}\]
 with $f:\bbc\to\bbc,\,z\mapsto e^{i\mu z}\frac{q(z)}{p(z)}$ and $\mathrm{Res}_zf$ denoting the residual of $f$ at point $z$. 
 
 Turning now to our function $g$, note first that it is well-defined by virtue of \cite[Lemma 3.10]{Marquardtetal2007}. It is clearly continuous and we have from elementary matrix theory that
 \[g(\mu)=P(i\mu)^{-1}Q(i\mu)=\frac{S(i\mu)}{\mathrm{det}(P(i\mu))}\]
 where $S:\bbc\to M_d(\bbc)$ is some matrix-valued polynomial. Using \cite[Lemma 3.11]{Marquardtetal2007} it is easy to see that the complex-valued polynomial $\mathrm{det}(P(i\mu))$ in $\mu$ is of 
 higher degree than $S(i\mu)$. Since all zeros of $P$ are assumed to have non-vanishing real part, the zeros of $P(i\hspace{0.5mm}\cdot\hspace{0.5mm})$ have non-vanishing imaginary part. On the 
 one hand this implies that all components of the function $g$ are continuously differentiable and hence $g$ is of bounded variation on compacts. 
 
 On the other hand this enables us to apply the above stated results from complex function theory component wise and we deduce for all $j,k=1,\ldots,d$ and $\mu\in\bbr,\,\mu\neq 0$,
 \begin{align}
  \left(\sqrt{2\pi}h(\mu)\right)_{jk}
   &=\left(\int_{-\infty}^\infty e^{i\mu t} g(t)\,dt\right)_{jk} \nonumber \\
   &=2\pi i\bigg(\mathds{1}_{\left\{\mu>0\right\}}\cdot\sum\limits_{\substack{z\in\bbc:\,\mathrm{Im}(z)>0, \\ \mathrm{det}(P(iz))=0}}\mathrm{Res}_zf_{jk}
    - \mathds{1}_{\left\{\mu<0\right\}}\cdot\sum\limits_{\substack{z\in\bbc:\,\mathrm{Im}(z)<0,\\ \mathrm{det}(P(iz))=0}}\mathrm{Res}_zf_{jk}\bigg) \nonumber \\
   &=2\pi i\bigg(\mathds{1}_{\left\{\mu>0\right\}}\cdot\sum\limits_{\substack{z\in\bbc:\,\mathrm{Re}(z)<0, \\ \mathrm{det}(P(z))=0}}\mathrm{Res}_{-iz}f_{jk} 
    - \mathds{1}_{\left\{\mu<0\right\}}\cdot\sum\limits_{\substack{z\in\bbc:\,\mathrm{Re}(z)>0, \\ \mathrm{det}(P(z))=0}}\mathrm{Res}_{-iz}f_{jk}\bigg) \label{Equation 4.3}
 \end{align}
 where $f_{jk}:\bbc\to\bbc,\,z\mapsto e^{i\mu z}\frac{S_{jk}(iz)}{\mathrm{det}(P(iz))}$. 
 
 Let $\la$ denote the distinct zeros of $\mathrm{det}(P(z))$ and $m(\la)$ the multiplicity of the zero $\la$. Since it is well-known that the residual of any meromorphic function $f:\bbc\to\bbc$ at a pole 
 $a$ of order $n\in\bbn$ is given by
 \[\mathrm{Res}_af=\frac{1}{(n-1)!}\left[\frac{d^{n-1}}{dz^{n-1}}(z-a)^nf(z)\right]_{z=a}\]
 (see \cite[Section III.6, Remark 6.4.1]{Freitagetal2000}), the residual of $f_{jk}$ at the point $-i\la$, with $\la$ being any zero of $\mathrm{det}(P(z))$, can be written as
 \begin{align}
  \mathrm{Res}_{-i\la}f_{jk}
   &=\frac{1}{(m(\la)-1)!}\left[\frac{d^{m(\la)-1}}{dz^{m(\la)-1}}(z+i\la)^{m(\la)}e^{i\mu z}\frac{S_{jk}(iz)}{\mathrm{det}(P(iz))}\right]_{z=-i\la}
    = -i\cdot\sum\limits_{s=0}^{m(\la)-1}c_{\la s}^{jk}\mu^se^{\la\mu} \label{Equation 4.4}
 \end{align}
 for appropriate complex constants $c_{\la s}^{jk}$ (where the sum reduces to $-i\cdot S_{jk}(\la)/\left[\frac{d}{dz}\,\mathrm{det}(P(z))\right]_{z=\la}e^{\la\mu}$ if $m(\la)=1$).
 Thus (\ref{Equation 4.3}) becomes
 \begin{align}
  h(\mu)_{jk}
   &=\sqrt{2\pi}\Big(\mathds{1}_{\left\{\mu>0\right\}}\cdot\sum\limits_{\la:\,\mathrm{Re}(\la)<0}\sum\limits_{s=0}^{m(\la)-1}c_{\la s}^{jk}\mu^s e^{\la\mu} 
    - \mathds{1}_{\left\{\mu<0\right\}}\cdot\sum\limits_{\la:\,\mathrm{Re}(\la)>0}\sum\limits_{s=0}^{m(\la)-1}c_{\la s}^{jk}\mu^s e^{\la\mu}\Big) \nonumber \\
   &=\sqrt{2\pi}\sum\limits_{\la}\sum_{s=0}^{m(\la)-1}\left(\wt{c}_{\la s}^{jk}\mu^se^{\la\mu}\mathds{1}_{\left\{\mathrm{Re}(\la)\cdot\mu<0\right\}}\right) \label{Equation 4.5}
 \end{align}
 with $\wt{c}_{\la s}^{jk}:=c_{\la s}^{jk}$ if $\mathrm{Re}(\la)<0$ and $\wt{c}_{\la s}^{jk}:=-c_{\la s}^{jk}$ if $\mathrm{Re}(\la)>0.$ 
 
 Hence $h$ is obviously in $L^\delta(M_d(\bbc))$ for any $\delta\geq 1$ and by virtue of the ``inversion formula'' (cf. \cite[p. 177]{Katznelson2004}) we obtain $\wh{h}=g$ in $L^2(M_d(\bbc))$.
 
 Finally, defining the $d\times d$ matrices $C_{\la s}:=(\wt{c}_{\la s}^{jk})_{jk}$, we obtain the claimed representation of $h$ in (\ref{Equation 4.2}). 
\end{proof}
Let us now give a spectral definition of regularly varying MCARMA processes with index $\al\in(1,2]$. Note that the well-definedness is ensured by Proposition \ref{Proposition 6.9} and Lemma \ref{Lemma 4.5}. 
Our definition is the regularly varying analogon of \cite[Definition 3.18]{Marquardtetal2007}, hence similar arguments to \cite[Remark 3.6 and 3.19]{Marquardtetal2007} show that a regularly varying MCARMA 
process $Y$ can be interpreted as a solution to the formal $p$-th-order $d$-dimensional differential equation 
\[P(D)Y_t = Q(D)DL_t,\quad t\in\bbr,\]
where $D$ denotes the differentiation operator with respect to $t$ and $P$ and $Q$ are the autoregressive and moving average polynomial, respectively. This differential equation is obviously comparable 
to the difference equation characterizing ARMA processes in discrete time.
\begin{Definition}[\textbf{Regularly Varying MCARMA Process}]
 Let $L=(L_t)_{t\in\bbr}$ be a $d$-dimensional L\'evy process with generating triplet $(\ga,\Sigma,\nu)$ where $\nu\in RV(\al, l, \kappa_\nu)$ with index $\al\in(1,2]$ and suppose $\bbE[L_1]=0$. 
 Let $M$ be again the corresponding random content of Theorem \ref{Theorem 6.4}. Then a $d$-dimensional {\rm regularly varying L\'evy-driven continuous time autoregressive moving average process}
 $(Y_t)_{t\in\bbr}$ of order $(p,q)$ with $p,\,q\in\bbn_0,\,p>q$ ({\rm regularly varying MCARMA$(p,q)$ process}) with index $\al$ is defined as the regularly varying process 
 \begin{align}
  Y_t:&=\mathop{\bbP-\lim}\limits_{\la\to\infty}\int_{-\la}^\la e^{it\mu}P(i\mu)^{-1}Q(i\mu)\left(1-\frac{\left|\mu\right|}{\la}\right) M(d\mu),\quad t\in\bbr,\quad\text{where} \nonumber \\
  P(z):&=z^p\mathrm{I}_d + z^{p-1}A_1 + \ldots + A_p\quad\text{and} \nonumber \\
  Q(z):&=z^qB_0+z^{q-1}B_1+\ldots+B_q \nonumber
 \end{align}
 are the autoregressive and moving average polynomial, respectively. \\
 Here $A_i,\,B_j\in M_d(\bbr)$ are real matrices satisfying $B_0\neq 0$ and $\scrn(P)=\left\{z\in\bbc:\,\mathrm{det}(P(z))=0\right\}\subseteq\bbr\backslash\left\{0\right\}+i\bbr$ and $\kappa_\nu$ is a 
 Radon measure such that $\kappa_\nu\left(h^{-1}(s)(\bbc^d\backslash\left\{0\right\})\cap\bbr^d\right)=0$ does not hold for almost every $s$, where $h=\wh{g}(-\hspace{0.5mm}\cdot\hspace{0.5mm})$ with
 $g=P(i\hspace{0.5mm}\cdot\hspace{0.5mm})^{-1}Q(i\hspace{0.5mm}\cdot\hspace{0.5mm})$ (cf. (\ref{Equation 4.2})).
\label{Definition 6.10}
\end{Definition}
\noindent
Note that in the causal case (i.e. $\scrn(P)\subseteq (-\infty,0)+i\bbr$) with $p=q+1$ it is sufficient that $B_0$ is invertible in order to ensure that $\kappa_\nu$ satisfies the preceding condition.

In addition to their spectral representation (in the summability sense), the following moving average representation of regularly varying MCARMA processes is immediately obtained. 
\begin{Corollary}
 Let $Y=(Y_t)_{t\in\bbr}$ be a regularly varying MCARMA$(p,q)$ process of index $\al\in(1,2]$. Then $Y$ has the moving average representation 
 \[Y_t=\frac{1}{\sqrt{2\pi}}\int_{-\infty}^\infty h(t-\mu)\,L(d\mu)\]
 for all $t\in\bbr$, where the kernel function (cf. (\ref{Equation 4.2})) is given by $h=\frac{1}{\sqrt{2\pi}}\int_{-\infty}^\infty e^{is\hspace{0.5mm}\cdot\hspace{0.5mm}}P(is)^{-1}Q(is)\,ds$.
\label{Corollary 6.11}
\end{Corollary}
\begin{proof}[\textbf{Proof.}]
 Combine Proposition \ref{Proposition 6.9} and Lemma \ref{Lemma 4.5}.
\end{proof}
Using this kernel representation, strict stationarity of regularly varying MCARMA processes is obtained by applying \cite[Theorem 4.3.16]{Applebaum2009}.
\begin{Proposition}
 The regularly varying MCARMA process is strictly stationary.
\label{Proposition 6.17}
\end{Proposition}
\begin{Remark}$~~$
 \begin{enumerate}
  \item
   One might think that regularly varying MCARMA processes also have a bona fide spectral representation of the form
   \[Y_t=\int_{-\infty}^\infty e^{it\mu}P(i\mu)^{-1}Q(i\mu)\,\wt{M}(d\mu),\quad t\in\bbr,\] 
   for an appropriate extension $\wt{M}$ of the regularly varying random content of Theorem \ref{Theorem 6.4} to $\scrb_0(\bbr)$, the collection of all Borel sets with finite Lebesgue measure. 
   In the case of driving L\'evy processes that are symmetric $\al$-stable ($S\al S$ for short) with index of stability $\al\in(1,2)$, the relationship between harmonizable $S\al S$ processes (i.e. Fourier 
   transforms of possibly dependently scattered $S\al S$ noises) and moving averages of stationarily and independently scattered $S\al S$ noises has been studied for a long time. Finally, it has been shown 
   in \cite[Proposition 1.9]{Makagonetal1990} that if the $S\al S$ MCARMA process with $\al\in(1,2)$ had such a bona fide spectral representation, it would be equal to $0$ for all $t\in\bbr$. 
   Thus such a representation cannot exist in general.
  \item
   However, for $\al=2$ we can distinguish the following two cases: if $\bbE[\left\|L_1\right\|^2]<\infty$, then we are in the setting of \cite{Marquardtetal2007} and one can derive a bona fide spectral 
   representation for the driving L\'evy and the associated MCARMA process. If $L_1$ has infinite variance, then the $L^2$-theory is not applicable but we get a spectral representation (in the summability sense) 
   for the driving L\'evy and the associated MCARMA process according to Theorem \ref{Theorem 6.4} and Definition \ref{Definition 6.10}, respectively.
 \end{enumerate}
\label{Remark 4.12}
\end{Remark}
\section{Consistency to Previously Defined Causal MCARMA Processes}
The established spectral representation clearly extends the L\'evy-case with finite second moments. However, also in \cite{Marquardtetal2007} an extension of MCARMA processes in the causal case, 
the so-called causal MCARMA process, has been introduced. 
In this section we are going to prove that both definitions coincide when they apply both.
\begin{Proposition}
 Let the polynomials $P$ and $Q$ be defined as in Lemma \ref{Lemma 4.5}. Moreover we define the matrices
 \[A:=\begin{pmatrix}
  0      & \mathrm{I}_d & 0            & \ldots & 0            \\
  0      & 0            & \mathrm{I}_d & \ddots & \vdots       \\
  \vdots & \vdots       & \ddots       & \ddots & 0            \\
  0       & 0            & \ldots       & 0      & \mathrm{I}_d \\
  -A_p   & -A_{p-1}     & \ldots       & \ldots & -A_1 \end{pmatrix}
 \in M_{dp}(\bbc)\]
 and $\beta=(\beta_1^*,\ldots,\,\beta_p^*)^*\in M_{dp\times d}(\bbc)$ with
 \beqq
  \beta_{p-j}:=-\sum\limits_{i=1}^{p-j-1} A_i\beta_{p-j-i} + B_{q-j},\quad j=0,1,\ldots,p-1,
 \label{Equation 2.8}
 \eeqq
 setting $B_i=0$ for $i<0$. Then, for any $t\in\bbr$, 
 \beqq
  \left(\mathrm{I}_d,0_{M_d(\bbc)},\ldots,0_{M_d(\bbc)}\right)e^{tA}\beta=\frac{1}{2\pi i}\int_\rho e^{tz}P(z)^{-1}Q(z)\, dz
 \label{Equation 4.7}
 \eeqq
 where $\rho$ is a simple closed curve in the complex plane that encircles all eigenvalues of the matrix $A$. 
\label{Proposition 4.9}
\end{Proposition}
\begin{proof}[\textbf{Proof.}]
 By virtue of \cite[Proposition 11.2.1]{Bernstein2005} we have $e^{tA}=\frac{1}{2\pi i}\int_\rho e^{tz}(z\mathrm{I}_{dp}-A)^{-1}\,dz$ for any $t\in\bbr$ with $\rho$ being a simple closed curve in the 
 complex plane enclosing the spectrum of $A$. Setting
 \beqq
  h_{k,p}(z):=\sum\limits_{u=0}^{p-k}z^uA_{p-k-u},\quad k=1,\ldots,p,
 \label{Equation 4.8}
 \eeqq
 with $A_0:=\mathrm{I}_d$, and 
 \[r_{k}(z):=-\sum\limits_{u=0}^{k}z^uA_{p-u},\quad k=0,1,\ldots,p-2,\]
 one verifies that, for all $z$ outside of the spectrum of $A$, the $d\times d$ blocks $c_{ij}(z),\,i,j=1,\ldots,p$, of the matrix $(z\mathrm{I}_{dp}-A)^{-1}\in M_{dp}(\bbc)$ are given by
 \[c_{ij}(z)=P(z)^{-1}\left\{\begin{array}{l} z^{i-1}h_{j,p}(z),\hfill\text{if }i\leq j, \\ z^{i-j-1}r_{j-1}(z),\quad\hfill\text{if }i>j.\end{array}\right.\]
 Indeed, one can show by simple calculations that this matrix is a left inverse for $z\mathrm{I}_{dp}-A$ and thus, due to \cite[Corollary 2.6.4]{Bernstein2005}, it is the unique inverse of $z\mathrm{I}_{dp}-A$. Hence
 \[\left(\mathrm{I}_d,0_{M_d(\bbc)},\ldots,0_{M_d(\bbc)}\right)e^{tA}\beta=\frac{1}{2\pi i}\int_\rho e^{tz}P(z)^{-1}\cdot\bigg(\sum\limits_{j=1}^p h_{j,p}(z)\beta_j\bigg)\,dz.\]
 Since $B_i=0$ for all $i<0$, we obtain
 \begin{align}
  \sum\limits_{j=1}^p h_{j,p}(z)\beta_j
   &\stackrel{(\scriptsize\ref{Equation 4.8})}{=}\sum\limits_{j=1}^p\sum\limits_{u=0}^{p-j}z^uA_{p-j-u}\beta_j = \sum\limits_{u=0}^{p-1}z^u\bigg(\sum\limits_{j=1}^{p-u} A_{p-j-u}\beta_j\bigg)
    =\sum\limits_{u=0}^{p-1}z^u\bigg(\sum\limits_{j=1}^{p-u-1}A_j\beta_{p-j-u} + \beta_{p-u}\bigg) \nonumber \\
   &\stackrel{(\scriptsize\ref{Equation 2.8})}{=}\sum\limits_{u=0}^{p-1} z^uB_{q-u} = \sum\limits_{u=0}^q z^uB_{q-u} = Q(z) \nonumber
 \end{align}
 and thus (\ref{Equation 4.7}) is shown.
\end{proof}
\noindent
Now observe that the causal MCARMA process can be represented as (cf. \cite[Theorem 3.12]{Marquardtetal2007})
\[Y_t=\int_{-\infty}^t\left(\mathrm{I}_d,0_{M_d(\bbc)},\ldots,0_{M_d(\bbc)}\right)e^{(t-s)A}\beta\,L(ds),\quad t\in\bbr.\]
Since $\scrn(P)\subseteq (-\infty,0)+i\bbr$, we obtain for all $s<t$, due to Proposition \ref{Proposition 4.9} and the Residue Theorem,
\begin{align}
 \left(\mathrm{I}_d,0_{M_d(\bbc)},\ldots,0_{M_d(\bbc)}\right)e^{(t-s)A}\beta
  &=\frac{1}{2\pi i}\int_\rho \underbrace{e^{(t-s)z}P(z)^{-1}Q(z)}_{=:f(z)} dz = \sum\limits_{\substack{z\in\bbc:\,\mathrm{Re}(z)<0,\\ \mathrm{det}(P(z))=0}}\left(\mathrm{Res}_z f_{jk}\right)_
   {j,k=1,\ldots,d} \nonumber \\
  &=i\cdot\sum\limits_{\substack{z\in\bbc:\,\mathrm{Re}(z)<0,\\ \mathrm{det}(P(z))=0}}\left(\mathrm{Res}_{-iz} \wt{f}_{jk}\right)_{j,k=1,\ldots,d} \nonumber
\end{align}
with $\wt{f}(z):=e^{i(t-s)z}P(iz)^{-1}Q(iz),\,z\in\bbc$. Using then (\ref{Equation 4.4}) and (\ref{Equation 4.5}), we deduce
\[\left(\mathrm{I}_d,0_{M_d(\bbc)},\ldots,0_{M_d(\bbc)}\right)e^{(t-s)A}\beta\cdot\mathds{1}_{\left\{s<t\right\}}=\frac{1}{\sqrt{2\pi}}h(t-s)\]
and the claimed consistency follows from Corollary \ref{Corollary 6.11}.
\begin{Remark}
 Similar results in the univariate case can be found in \cite[Lemma 2.3]{Brockwelletal2009}. We have shown that equation (2.10) in that lemma extends to its multivariate version (\ref{Equation 4.7}).
\label{Remark 4.10}
\end{Remark}
\begin{appendix}
\section{Fourier Transforms on the Real Line}
\label{Section 1.3}
In this appendix we summarize important results for Fourier transforms on the real line.

The theory for complex-valued functions in $L^1$ and $L^2$, resp., is rather standard. Also the extension to complex-valued functions in $L^p$ with $p\in(1,2)$, for which the Fourier transforms can be defined by
continuity as functions in $L^q$ with $q=p/(p-1)$, is quite common. For good expositions we refer the reader to \cite{Katznelson2004} or \cite{Steinetal1971}. In the following, let us state the multivariate versions of
some well-known univariate results.

For $p\in[1,2]$ we set
\[L^p(M_d(\bbc)):=\left\{f:\bbr\to M_d(\bbc)\text{ measurable},\ \int_\bbr\left\|f(t)\right\|^p dt<\infty\right\}.\]
It is independent of the norm $\left\|\hspace{0.5mm}\cdot\hspace{0.5mm}\right\|$ on $M_d(\bbc)$ (since $M_d(\bbc)$ is a finite dimensional linear space) and it is equal to the space of functions 
$f=(f_{ij}): \bbr\to M_d(\bbc)$ where all components $f_{ij},\,i,j=1,\ldots,d$, are in the usual space $L^p(\bbc)$. We equip $L^p(M_d(\bbc))$ with the norm
$\left\|f\right\|_{L^p}:=\left(\int_\bbr\left\|f(t)\right\|^p dt\right)^{1/p}$.

The Fourier transform $\wh{f}$ of $f\in L^1(M_d(\bbc))$ is defined by 
\[\wh{f}(\xi):=\frac{1}{\sqrt{2\pi}}\int_\bbr e^{-i\xi x}f(x)\,dx,\quad\xi\in\bbr.\]
Clearly, $\wh{f}:\bbr\to M_d(\bbc)$ and $\wh{f}$ can be interpreted as the component wise Fourier transformation of $f_{ij},\,i,j=1,\ldots,d$, i.e. $\wh{f}=(\wh{f}_{ij})$. This is why, as in the univariate case, 
the Fourier transform $\wh{f}$ of $f\in L^p(M_d(\bbc)),\,1<p\leq 2$, can be defined as the limit in $L^q(M_d(\bbc)),\,q=p/(p-1)$, of the sequence 
$\frac{1}{\sqrt{2\pi}}\int_{-n}^n e^{-ix\hspace{0.5mm}\cdot\hspace{0.5mm}} f(x)\,dx$ as $\nto$. The mapping $f\mapsto\wh{f}$ is a linear continuous operator from $L^p(M_d(\bbc))$ into $L^q(M_d(\bbc))$.

Recall that for $f\in L^1(\bbc)$ and $g\in L^p(M_d(\bbc)),\,p\in[1,2]$, the convolution $h(x):=(f*g)(x):=\int_\bbr f(x-y)g(y)\,dy$ is Lebesgue-a.e. well-defined, $h\in L^p(M_d(\bbc))$ and its Fourier transform
is given by $\wh{h} = \sqrt{2\pi}\wh{f}\cdot\wh{g}$.

Moreover, we define $F_\la(x):=\la\cdot F(\la x),\ \la\in(0,\infty),$ where
\[F(x):=\frac{1}{2\pi}\left(\frac{\sin{x/2}}{x/2}\right)^2=\frac{1}{2\pi}\int_\bbr  e^{i\xi x}(1-\left|\xi\right|)\mathds{1}_{[-1,1]}(\xi)\,d\xi.\]
The family $(F_\la)_{\la\in(0,\infty)}$ is called F\'ejer kernel on the real line. It is easy to show that, for almost all $\xi\in\bbr$,
\beqq
 \wh{F_\la}(\xi) = \frac{1}{\sqrt{2\pi}}\left(1-\frac{\left|\xi\right|}{\la}\right)\mathds{1}_{[-\la,\la]}(\xi)=:\Delta_\la(\xi)
\label{Equation 1.12}
\eeqq
and thus 
\beqq
 \wh{F_\la*g}=\sqrt{2\pi}\wh{F_\la}\cdot\wh{g}=\sqrt{2\pi}\Delta_\la\cdot\wh{g}
\label{Equation 1.19}
\eeqq
for all $g\in L^p(M_d(\bbc)),\,p\in[1,2]$. Using these properties one can show the following central result in classical Fourier theory (cf. \cite[Chapter VI]{Katznelson2004}):
\begin{Theorem}
 Let $f\in L^p(M_d(\bbc))$ with $p\in[1,2]$ and $(F_\la)_{\la\in(0,\infty)}$ be the F\'ejer kernel on $\bbr$. Then 
 \[\lim\limits_{\la\to\infty}\left\|f-F_\la*f\right\|_{L^p}=0\]
 showing that $F_\la$ is an approximate convolution identity.
\label{Theorem 1.36}
\end{Theorem}
\end{appendix}
%
%%%%%%%%%%%%%%%%%  Danksagungen  %%%%%%%%%%%%%%%%%%%%%%%%%%%%%%%%%%%
\section*{Acknowledgments}
Both authors are grateful for financial support of the Technische Universit\"at M\"unchen - Institute for Advanced Study, funded by the German Excellence Initiative. 
Moreover Florian Fuchs gratefully acknowledges the support of the TUM Graduate School's International School of Applied Mathematics.
%
%%%%%%%%%%%%%%%%%  Literaturverzeichnis  %%%%%%%%%%%%%%%%%%%%%%%%%%%
%{\small\bibliography{../../Promo}{}}

\begin{thebibliography}{10}

\bibitem{Applebaum2009}
{\sc Applebaum, D.}
\newblock {\em L\'evy Processes and Stochastic Calculus}, 2nd~ed., vol.~116 of
  {\em Cambridge Studies in Advanced Mathematics}.
\newblock Cambridge University Press, Cambridge, UK, 2009.

\bibitem{Arato1982}
{\sc Arat\'o, M.}
\newblock {\em Linear Stochastic Systems with Constant Coefficients}, vol.~45
  of {\em Lecture Notes in Control and Information Sciences}.
\newblock Springer-Verlag, Berlin, 1982.

\bibitem{Bernstein2005}
{\sc Bernstein, D.~S.}
\newblock {\em Matrix Mathematics: Theory, Facts, and Formulas with Application
  to Linear Systems Theory}.
\newblock Princeton University Press, Princeton, New Jersey, 2005.

\bibitem{Bretagnolle1972}
{\sc Bretagnolle, J.}
\newblock {\em $p$-variation de fonctions al\'eatoires}, vol.~258 of {\em Lecture
  Notes in Mathematics}.
\newblock Springer-Verlag, Berlin, 1972.

\bibitem{Brockwell2001}
{\sc Brockwell, P.~J.}
\newblock L\'evy-driven {CARMA} processes.
\newblock {\em Ann. Inst. Statist. Math. 53\/} (2001), 113--124.

\bibitem{Brockwell2004}
{\sc Brockwell, P.~J.}
\newblock Representations of continuous-time {ARMA} processes.
\newblock {\em J. Appl. Probab. 41A\/} (2004), 375--382.

\bibitem{Brockwelletal1991}
{\sc Brockwell, P.~J., and Davis, R.~A.}
\newblock {\em Time Series: Theory and Methods}, 2nd~ed.
\newblock Springer-Verlag, New York, 1991.

\bibitem{Brockwelletal2009}
{\sc Brockwell, P.~J., and Lindner, A.}
\newblock Existence and uniqueness of stationary {L}\'evy-driven {CARMA}
  processes.
\newblock {\em Stochastic Process. Appl. 119\/} (2009), 2660--2681.

\bibitem{Cambanisetal1987}
{\sc Cambanis, S., Hardin{, Jr.}, C.~D., and Weron, A.}
\newblock Ergodic properties of stationary stable processes.
\newblock {\em Stochastic Process. Appl. 24\/} (1987), 1--18.

\bibitem{Cambanisetal1993}
{\sc Cambanis, S., and Houdr\'e, C.}
\newblock Stable processes: moving averages versus {F}ourier transforms.
\newblock {\em Probab. Theory Related Fields 95\/} (1993), 75--85.

\bibitem{Cambanisetal1989b}
{\sc Cambanis, S., and Miamee, A.~G.}
\newblock On prediction of harmonizable stable processes.
\newblock {\em Sankhy{\=a} Ser. A 51\/} (1989), 269--294.

\bibitem{Cambanisetal1984}
{\sc Cambanis, S., and Soltani, A.~R.}
\newblock Prediction of stable processes: Spectral and moving average
  representations.
\newblock {\em Z. Wahrsch. Verw. Gebiete 66\/} (1984), 593--612.

\bibitem{Daleyetal2003}
{\sc Daley, D.~J., and Vere-Jones, D.}
\newblock {\em An Introduction to the Theory of Point Processes, Volume I:
  Elementary Theory and Methods}, 2nd~ed.
\newblock Probability and its Applications. Springer-Verlag, New York, 2003.

\bibitem{Doob1953}
{\sc Doob, J.~L.}
\newblock {\em Stochastic Processes}.
\newblock John Wiley \& Sons, New York, 1953.

\bibitem{Freitagetal2000}
{\sc Freitag, E., and Busam, R.}
\newblock {\em Funktionentheorie I}, 3rd~ed.
\newblock Springer-Verlag, Berlin, 2000.

\bibitem{Garciaetal2010}
{\sc Garc\'ia, I., Kl\"uppelberg, C., and M\"uller, G.}
\newblock Estimation of stable {CARMA} models with an application to
  electricity spot prices.
\newblock Statistical Modelling, to appear, available from
  http://www-m4.ma.tum.de/Papers, 2010.

\bibitem{Hosoya1982}
{\sc Hosoya, Y.}
\newblock Harmonizable stable processes.
\newblock {\em Z. Wahrsch. Verw. Gebiete 60\/} (1982), 517--533.

\bibitem{Hultetal2006}
{\sc Hult, H., and Lindskog, F.}
\newblock On regular variation for infinitely divisible random vectors and
  additive processes.
\newblock {\em Adv. in Appl. Probab. 38\/} (2006), 134--148.

\bibitem{Katznelson2004}
{\sc Katznelson, Y.}
\newblock {\em An Introduction to Harmonic Analysis}, 3rd~ed.
\newblock Cambridge University Press, Cambridge, UK, 2004.

\bibitem{Lang1993}
{\sc Lang, S.}
\newblock {\em Complex Analysis}, 3rd~ed., vol.~103 of {\em Graduate Texts in
  Mathematics}.
\newblock Springer-Verlag, New York, 1993.

\bibitem{Larssonetal2006}
{\sc Larsson, E.~K., Mossberg, M., and S\"oderstr\"om, T.}
\newblock An overview of important practical aspects of continuous-time {ARMA}
  system identification.
\newblock {\em Circuits Systems Signal Process. 25\/} (2006), 17--46.

\bibitem{Lindskog2004}
{\sc Lindskog, F.}
\newblock {\em Multivariate Extremes and Regular Variation for Stochastic
  Processes}.
\newblock PhD thesis, ETH Zurich, 2004.

\bibitem{Makagonetal1990}
{\sc Makagon, A., and Mandrekar, V.}
\newblock The spectral representation of stable processes: Harmonizability and
  regularity.
\newblock {\em Probab. Theory Related Fields 85\/} (1990), 1--11.

\bibitem{Marquardtetal2007}
{\sc Marquardt, T., and Stelzer, R.}
\newblock Multivariate {CARMA} processes.
\newblock {\em Stochastic Process. Appl. 117\/} (2007), 96--120.

\bibitem{Moseretal2010}
{\sc Moser, M., and Stelzer, R.}
\newblock Tail behavior of multivariate {L}\'evy-driven mixed moving average
  processes and sup{OU} stochastic volatility models.
\newblock Available from http://www-m4.ma.tum.de/Papers (preprint), 2010.

\bibitem{Rajputetal1989}
{\sc Rajput, B.~S., and Rosi\'nski, J.}
\newblock Spectral representations of infinitely divisible processes.
\newblock {\em Probab. Theory Related Fields 82\/} (1989), 451--487.

\bibitem{Resnick2007}
{\sc Resnick, S.~I.}
\newblock {\em Heavy-Tail Phenomena: Probabilistic and Statistical Modeling}.
\newblock Springer-Verlag, New York, 2007.

\bibitem{Rootzen1978}
{\sc Rootz\'en, H.}
\newblock Extremes of moving averages of stable processes.
\newblock {\em Ann. Probab. 6\/} (1978), 847--869.

\bibitem{Rozanov1967}
{\sc Rozanov, Y.~A.}
\newblock {\em Stationary Random Processes}.
\newblock Holden - Day, San Francisco, 1967.

\bibitem{Samorodnitskyetal1994}
{\sc Samorodnitsky, G., and Taqqu, M.~S.}
\newblock {\em Stable Non-Gaussian Random Processes: Stochastic Models with
  Infinite Variance}.
\newblock Chapman \& Hall/CRC, Florida, 1994.

\bibitem{Sato1999}
{\sc Sato, K.}
\newblock {\em L\'evy Processes and Infinitely Divisible Distributions}, vol.~68
  of {\em Cambridge Studies in Advanced Mathematics}.
\newblock Cambridge University Press, Cambridge, UK, 1999.

\bibitem{Steinetal1971}
{\sc Stein, E.~M., and Weiss, G.}
\newblock {\em Introduction to Fourier Analysis on Euclidean Spaces}.
\newblock Princeton University Press, Princeton, New Jersey, 1971.

\bibitem{Tauchenetal2006}
{\sc Tauchen, G., and Todorov, V.}
\newblock Simulation methods for {L}\'evy-driven continuous-time autoregressive
  moving average ({CARMA}) stochastic volatility models.
\newblock {\em J. Bus. Econom. Statist. 24\/} (2006), 455--469.

\end{thebibliography}
%

%
\end{document}